\documentclass[UTF8]{amsart}
\usepackage{amsmath,amssymb,graphicx,bbm,amsthm}
\usepackage{caption}
\usepackage{subcaption}
\usepackage[american]{babel}
\usepackage{color}
\usepackage[all]{xy}
\newtheorem{theorem}{Theorem}[section]
\newtheorem{definition}[theorem]{Definition}
\newtheorem{lemma}[theorem]{Lemma}
\newtheorem{proposition}[theorem]{Proposition}

\newtheorem{cor}[theorem]{Corollary}

\newtheorem{remark}[theorem]{Remark}
\newtheorem{question}[theorem]{Question}
\numberwithin{equation}{section}

\DeclareMathOperator{\sys}{sys}
\DeclareMathOperator{\inj}{inj}
\DeclareMathOperator{\Ext}{Ext}
\DeclareMathOperator{\Crit}{Crit}
\DeclareMathOperator{\arccosh}{arccosh}
\DeclareMathOperator{\arcsinh}{arcsinh}
\DeclareMathOperator{\diam}{diam}
\usepackage{xcolor}
\usepackage{verbatim}

\usepackage{listings}
\usepackage{color} 
\definecolor{mygreen}{RGB}{28,172,0} 
\definecolor{mylilas}{RGB}{170,55,241}

\begin{document}

\lstset{language=Matlab,
    breaklines=true,
    morekeywords={matlab2tikz},
    keywordstyle=\color{blue},
    morekeywords=[2]{1}, keywordstyle=[2]{\color{black}},
    identifierstyle=\color{black},
    stringstyle=\color{mylilas},
    commentstyle=\color{mygreen},
    showstringspaces=false,
    numbers=left,
    numberstyle={\tiny \color{black}},
    numbersep=9pt, 
	xleftmargin=1cm
}

\title{Shape of filling-systole subspace in surface moduli space and critical points of systole function}

\author{Yue Gao}
\address{BICMR, Peking University, Beijing 100871, CHINA}
\email{yue\_gao@pku.edu.cn}
\date{}

\maketitle

\begin{abstract}

		This paper studies the space $X_g\subset \mathcal{M}_g$ consisting of surfaces with filling systoles and its subset, critical points of the systole function. In the first part, we obtain a surface with Teichm\"uller distance $\frac{1}{5}\log\log g$ to $X_g$ and
		in the second and third part, prove that most points in $\mathcal{M}_g$ have Teichm\"uller distance $\frac{1}{5}\log\log g$ to $X_g$ and Weil-Petersson distance $0.6521(\sqrt{\log g}-\sqrt{7\log\log g})$ respectively. 
		Therefore we prove that the radius-$r$ neighborhood of $X_g$ is not able to cover the thick part of $\mathcal M_g$ for any fixed $r>0$. 
		In last two parts, we get critical points with small and large (comparable to diameter of thick part of $\mathcal M_g$) distance respectively. 
\end{abstract}

\tableofcontents
\section{Introduction}
\subsection{Motivations}
A long-standing and difficult question on the moduli space of Riemann surface (denoted $\mathcal M_g$) is to construct a spine of $\mathcal M_g$ (i.e. the deformation retract of $\mathcal M_g$ with the minimal dimension \footnote{In some paper, a deformation retract of $\mathcal M_g$ is called a spine of $\mathcal M_g$ and the ones with minimal dimension  are called minimal (or optimal) spine}. )
This question is equivalent to construct mapping class group equivariant deformation retract with minimal dimension of the Teichm\"uller space $\mathcal T_g$. In an unpublished manuscript \cite{thurston1998minimal}, Thurston proposed a candidate for the spine of $\mathcal M_g$ (see \cite{anderson2016relative}). This candidate consists of surfaces whose shortest geodesics are filling, denoted by $X_g$ (A finite set of essential curves on a surface is filling if the curves cut the surface into polygonal disks. ). Thurston outlined a proof that $X_g$ is a deformation retract of $\mathcal M_g$ but the proof seems difficult to complete. 
Recently, some progress to the dimension of $X_g$ has been made, for example, a codimenstion $2$ deformation retract of  $\mathcal M_g$ containing $X_g$ (Ji \cite{ji2014well}) and a $4g-5$-cell contained in $X_g$ (Fortier-Bourque \cite{fortier2019hyperbolic}), But determining the dimension of $X_g$ seems to be still very difficult. 

Our work mainly concerns the shape of $X_g$ with respect to the Teim\"uller metric and Weil-Petersson metric on $\mathcal M_g$. 
Shape of $X_g$ was firstly studied by Anderson-Parlier-Pettet in \cite{anderson2016relative} and our work is partly inspired by the notion of sparseness of subsets in $\mathcal M_g$  they raised. 
Our question is 
\begin{question}
		Does there exist a number $R=R(g)>0$, such that for most points $p\in \mathcal M_g$, $d_{\mathcal T}(p,X_g)$ (or $d_{wp}(p,X_g)$ respectively) is larger than $R(g)$?
		\label{ques_sparse}
\end{question}
In other words, is $X_g$ in some sense 'sparse' in $\mathcal M_g$? 

In this paper, the main tool to deal with this question is the random surface theory with respect to Weil-Petersson volume. 

Another motivation to study the shape of $X_g$ is to understand the shape of critical point set of the systole function. 
On each surface $p\in \mathcal M_g$, systole is the length of the shortest geodesics on $p$. Therefore systole can be treated as a funtion on $\mathcal M_g$. Akrout showed this function is a topological Morse function in \cite{akrout2003singularites}, hence systole function has regular and critical points. The critical point set of this function is denoted by $\Crit(\sys_g)$. By a lemma by Schmutz Schaller \cite[Corollary 20]{schaller1999systoles}, $\Crit(sys_g)\subset X_g$. 
Therefore conclusions on the shape of $X_g$ implies corollaries on the shape of $\Crit(\sys_g)$. On the other hand, a natural question is to compare the difference of shape between $X_g$ and $\Crit(\sys_g)$. This program is closely related to the Mirzakhani's question if long finger exists. Details is in the next subsection.

\subsection{Results and perspectives}
Our first result is the construction of an example of surface distant from $X_g$. 

		\setcounter{section}{3}
		\setcounter{theorem}{6}
\begin{theorem}
		When $g\ge 3$, there is a surface $R_g$, its distance to $X_g$ is at least $\frac{1}{4} \log\left( \log{g} -K  \right)$. 
		where $K=\log 12$. 
		\label{thm_hole}
\end{theorem}

Before stating Theorem \ref{thm_random}, we make ``most points'' in Question \ref{ques_sparse} precise. 

The Weil-Petersson metric is a mapping class group equivariant Riemannian metric on the Teichm\"uller space, Therefore the volume of $\mathcal M_g$ and Borel subsets of $\mathcal M_g$ with respect to this metric is well defined. Mirzakhani invented the integration formula for geometric functions on $\mathcal M_g$ with respect to this volume in \cite{mirzakhani2007simple} and then calculate the volume of $\mathcal M_g$, she initiated a fast-growing area: random surface with respect to Weil-Petersson metric in \cite{mirzakhani2007simple} \cite{mirzakhani2013growth}. 

The random surface theory is based on the probability of Borel sets in $\mathcal M_g$. Mirzakhani defined the probability of a Borel set $B\subset \mathcal M_g$ as 
\[
		P_{WP}(B) = \frac{vol_{WP}(B)}{vol_{WP}(\mathcal M_g)}.
\]
Now we are ready to state Theorem \ref{thm_random}. 
		\setcounter{section}{4}
		\setcounter{theorem}{2}
\begin{theorem}
		When $g$ is sufficiently large, for the probability of a point $S$ in $\mathcal M_g$ whose Teichm\"uller distance to $X_g$ is smaller than $\frac{1}{5}\log \log g$ 
		\[
		P_{WP}\left( S|d_{\mathcal T}(S,X_g)<\frac{1}{5}\log\log g \right) \to 0		\]
		as $g\to \infty$. 

		\label{thm_random}
\end{theorem}
		\setcounter{section}{1}
		\setcounter{theorem}{1}

\begin{remark}
		The distance $\frac{1}{5} \log \log g$ is calculated from (\ref{for_dist_x}) in Lemma \ref{lem_dist_x} and the width in Nie-Wu-Xue's \cite[Theorem 2]{nie2020large}. Actually if we replace $\frac{1}{5}$ by any number smaller than $\frac{1}{4}$, this theorem still holds. Besides Lemma \ref{lem_dist_x} and \cite[Theorem 2]{nie2020large}, Theorem \ref{thm_random} also depends on Mirzakhani-Petri's \cite[Theorem 2.8]{mirzakhani2019lengths}. 
\end{remark}

		\setcounter{section}{4}
		\setcounter{theorem}{3}
Theorem \ref{thm_random} gives a positive answer to Question \ref{ques_sparse} with respect to Teichm\"uller distance. When $g$ is sufficiently large, most points in $\mathcal M_g$ has Teichm\"uller distance at least $\frac{1}{5}\log \log g$ to $X_g$. 

The moduli space $\mathcal M_g$ is divided into two parts. The thick part consists of surfaces with systole larger or equal to $\varepsilon$ for some fixed $\varepsilon>0$, denoted by $\mathcal M_g^{\ge\varepsilon}$. This part is compact in $\mathcal M_g$ and its diameter with respect to Teichm\"uller metric is $C\log \frac{g}{\varepsilon}$ for some $C>0$. 
The complementary part of the thick part is the thin part. 

By collar lemma (see for example \cite[Chapter 4]{buser2010geometry}), $X_g$ is contained in the thick part of $\mathcal M_g$ and we have
\begin{cor}

		When $g\to \infty$, 
		\[
				\frac{P_{WP}(S|d_{\mathcal T}(S,X_g)<\frac{1}{5}\log\log g ) }{P_{WP}(S\in \mathcal M_g^{\ge\varepsilon})}\to 0. 
\]
		\label{cor_thick}
\end{cor}

		\setcounter{section}{1}
		\setcounter{theorem}{2}

From Theorem \ref{thm_hole} or Corollary \ref{cor_thick}, the Hausdorff distance between thick part of $\mathcal M_g$ and $X_g$ is at least $\frac{1}{5}\log \log g$. 

The study on the shape of $X_g$ with respect to Teichm\"uller metric started from the paper \cite{anderson2016relative} by Anderson Parlier and Pettet. 
By comparing $X_g$ with $Y_g$, 
the subset of $\mathcal M_g$ with Bers' constant bounded above and below by constants, 
they obtained the following two results: (1) the diameter of $X_g$ is comparable with the thick part of  $\mathcal M_g$ \cite[Theorem 1.1]{anderson2016relative}; (2) the sparseness of $X_g\cap Y_g$ in $Y_g$: most points (according to a measure other than the Weil-Petersson volume defined on $Y_g$) in $Y_g$ has distance $\log g$ to $X_g\cap Y_g$ with respect to the metric induced by infimum length of path fully contained in $Y_g$ \cite[Theorem 1.3]{anderson2016relative}. 

The distance in our Theorem \ref{thm_hole} and \ref{thm_random} is smaller than the distance in \cite[Theorem 1.3]{anderson2016relative}, 
but we remove the restriction to $Y_g$ and obtain the sparseness of $X_g$ in $\mathcal M_g$ and thick part of $\mathcal M_g$. 

An immediate corollary to Theorem \ref{thm_hole} or Corollary \ref{cor_thick} is 

\begin{cor}
		For any $R>0$, when $g$ is sufficiently large, the $R$-neighborhood of $X_g$ does not cover the thick part of $\mathcal M_g$. Hence the $R$-neighborhood of $\Crit(\sys_g)$ does not cover the thick part of $\mathcal M_g$. 
		\label{cor_cover}
\end{cor}

For the thick part of $\mathcal M_g$, Fletcher Kahn and Markovic determined the minimal size of point set in $\mathcal M_g^{\ge\varepsilon}$, whose $R$ neighborhood covers the whole thick part for any $R>0$ \cite{fletcher2013moduli}. The size is $(Cg)^{2g}$ for $C=C(\varepsilon,R)>0$. Currtently the size of $\Crit(\sys_g)$ is not determined yet, but a known lower bound for $|\Crit(\sys_g)|$ given by Euler characteristic of $\mathcal M_g$ \cite{harer1986euler} is quite close to this number. However, by Corollary \ref{cor_thick} and \ref{cor_cover}, $\Crit(\sys_g)$ is sparse in $\mathcal M_g^{\ge\varepsilon}$. 

We also answer Question \ref{ques_sparse} with respect to Weil-Petersson metric
\setcounter{section}{5}
\setcounter{theorem}{6}

 \begin{theorem}
		 For the probability of a point in $\mathcal M_g$, whose Weil-Petersson distance to $X_g$ is smaller than $0.6521 ( \sqrt{\log g} - \sqrt{7\log\log g})$, we have
		\[
				P_{WP}\left(S|d_{wp}(S,X_g)<0.6521 \left( \sqrt{\log g} - \sqrt{7\log\log g}\right)\right) \to 0
		\]
		as $g\to \infty$. 
		 \label{thm_random_wp}
 \end{theorem}

\setcounter{section}{1}
\setcounter{theorem}{3}

Besides the tools used in the proof of Theorem \ref{thm_random}, to prove this theorem, we also use Wu's estimate of lower bounds of Weil-Petersson distance in \cite{wu2020new}. In \cite[Theorem 1.4]{wu2020new}, using this estimate, Wu has obtained that probability of Weil-Petersson $\sqrt{\log g}$-neighborhood of all surfaces with $o(\log g)$ Bers' constant tends to $0$ as $g$ tends to infinity. 

After answering Question \ref{ques_sparse}, a further question is that 
\begin{question}
		Is there  a critical point $p\in \Crit(\sys_g)$ and a large number $R(g)$ such that $B(p,R(g))$ contains no critical point except $p$? 
		\label{ques_finger}
\end{question}

This question is about the difference between $\Crit(\sys_g)$ in $X_g$. The radius gives a lower bound for the Hausdorff distance between $X_g$ and $\Crit(\sys_g)$. Moreover, Question \ref{ques_finger} is very close to but slightly weaker than Mirzakhani's question if there exists a long finger (see \cite{fortier2020local}) when the point $p$ is a local maximal point with large systole. 

For the local maximal point of the systole function $p$, a component of the level set $\left\{ q|\sys(q)>L \right\}$ that contains $p$ but does not contain any other critical point of the systole function is called a finger. The length of a finger is $\sys(p)-L$. If a finger is long, then the Teichm\"uller distance from $p$ to other critical point is large (at least $\frac{1}{2}\log\left( \sys(p)/L \right)$). 

We make the first attempt to compare the difference between $X_g$ and $\Crit(\sys_g)$. 

For any $g\ge2$, we take three surfaces $S_g^1$, $S_g^2$ and $S_g^3$ from \cite{anderson2011small} \cite{bai2019maximal} and \cite{fortier2020local} respectively. The surfaces $S_g^1$ and $S_g^3$ are known critical points and we prove $S_g^2$ is a critical point by our Proposition \ref{prop_euc}. Then we calculate the distance between the critical points. 
\setcounter{section}{8}
\setcounter{theorem}{2}

		\begin{theorem}
				For the surfaces $S_g^1,S_g^3\in \Crit(\sys_g)$, when $g\ge 13$, 
				\[
						d_{\mathcal{T}}(S_g^1,S_g^3) > \frac{1}{2} \log(g-6) -K,
				\]
				where $K=\frac{1}{2}\log \left( \frac{40}{3}\log\left( \frac{4g+4}{\pi} \right) \right)$. 
				\label{thm_dist_large}
		\end{theorem}
\setcounter{section}{1}
\setcounter{theorem}{3}
		Hence diameter of $\Crit(\sys_g)$ is comparable with the diameter of $X_g$ and diameter of thick part of $\mathcal M_g$. 

		On the other hand, the distance between $S_g^1,S_g^2$ is small. 
\setcounter{section}{7}
\setcounter{theorem}{7}
		\begin{theorem}
				For any $g\ge2$, $S_g^1,S_g^2\in \Crit(\sys_g)$, 
				\[
						d_{\mathcal T}(S_g^1,S_g^2)\le 2.3. 
				\]
				\label{thm_dist_small}
		\end{theorem}
\setcounter{section}{1}
\setcounter{theorem}{3}

		It is worth to mention that to prove the surface $\Sigma_g^2$ is a critical point, we use a conclusion (Proposition \ref{prop_euc}) that among all surface with a specific symmetry, the surface with maximal systole is a critical point. This proposition is a generalization to \cite[Theorem 37]{schaller1999systoles} and \cite[Proposition 6.3]{fortier2019hyperbolic}. The key point of this generalization is to construct a domain in $\mathcal M_g$ containing the point $p$ we consider, and $p$ is the maximal point of the systole function in the domain.

		\subsection{Methods}
		The distance between a surface $S\in\mathcal M_g$ and $X_g$ is bounded from below by Lemma \ref{lem_dist_x}: If for every filling curve set $F$ in which each pair of curves intersect at most once, $F$ contains a curve longer than $L$, then 
		\[
				d_{\mathcal T} (S,X_g)\ge \frac{1}{4}\log\frac{L}{\sys(S)}.
		\]

		We outline the construction of surface $S_g$ in Theorem \ref{thm_hole}. Our construction depends on the following observation: on a surface with a pants decomposition, a filling curve set must pass every pair of pants, otherwise the curve set is not filling. 

		The surface $S_g$ consists of isometric pairs of pants. There is a very special pair of pants $P$ in $S_g$ such that every non-separating simple closed curve passes $P$ has length at least $\log g$ while the systole of $S_g$ is a constant. 

		We take a trivalent tree $T$ (every vertex has degree $3$ except the leaves) with $n$ vertices and diameter comparable to $\log n$. The surface $S_g$ is obtained by replacing trivalent vertices of $T$ by a pair of pants and leaves of $T$ by one-holed tori. Thus the genus of $S_g$ 
		equals number of leaves of $T$ 
		and therefore is comparable to $n$. 
		By this construction, any non-separating curve in $S_g$ must pass some of the one-holed tori, otherwise the curve is contained in a $g$-holed sphere and thus separating. 

		For the tree $T$ with diameter $\log n$, there is a vertex  $v\in T$ such that the distance from $v$ to any leaf is comparable to $\log n$. Therefore 
		any non-separating curve passing the pants corresponding to $v$ has length comparable to 
		$\log n$ (or $\log g$),  
		while the systole of $S_g$ is a constant. 
		By Lemma \ref{lem_dist_x}, lower bound of $d_{\mathcal T} (S_g,X_g)$ is comparable with $\log\log g$. 

		The proof of Theorem \ref{thm_random} is by Lemma \ref{lem_dist_x} Nie-Wu-Xue's \cite[Theorem 2]{nie2020large} and Mirzakhani-Petri's \cite[Theorem 2.8]{mirzakhani2019lengths}. We outline the proof here and detail calculation is in Section \ref{sec_random}. 

		By \cite[Theorem 2]{nie2020large}, most surfaces in $\mathcal M_g$ contains a closed geodesic with a half collar of $\frac{1}{2}\log g-\log\log g$ width. This means any filling curve set on these surfaces contains a curve longer than the width. On the other hand, by \cite[Theorem 2.8]{mirzakhani2019lengths}, most surfaces have a systole shorter than $\log\log g$. Then by Lemma \ref{lem_dist_x}, Teichm\"uller distance between most points in $\mathcal M_g$ and $X_g$ is bounded from below. 

		For Theorem \ref{thm_random_wp}, the Weil-Petersson distance version, we observe that most surfaces contain a point with large injective radius (Lemma \ref{lem_inj_nwx}), while the corresponding point on surfaces in $X_g$ has an injective radius bounded from above by systole of the surface (Lemma \ref{lem_inj_fill}). Then by Wu's lower bound of Weil-Petersson distance by injective radius and systole (\cite[Theorem 1.1 and Corollary 1.2]{wu2020new}), and an argument similar to our Lemma \ref{lem_dist_x}, we proved Theorem \ref{thm_random_wp}. 

		Theorem \ref{thm_dist_large} is obtained by comparing the diameter of the two surfaces. This method is from \cite[Lemma 5.1]{rafi2013diameter}. 

		The shape of $S_g^1$ and $S_g^2$ is similar. Then we can construct the deformation from $S_g^1$ to $S_g^2$ explicitly. From the deformation we described in Section \ref{sec_small}, we calculate the distance and get Theorem \ref{thm_dist_small}.

		\subsection{Organization} In Section \ref{sec_pre}, we provide some priliminary knowledge on Teichm\"uller theory and the systole. Then we prove Theorem \ref{thm_hole} in Section \ref{sec_hole} and Theorem \ref{thm_random} in Section \ref{sec_random}. On the Weil-Petersson distance, we prove Theorem \ref{thm_random_wp} in \ref{sec_random_wp}. In Section \ref{sec_crit}, Proposition \ref{prop_euc} is proved. Then using Proposition \ref{prop_euc}, Theorem \ref{thm_dist_small}  is proved in Section \ref{sec_small}. Finally, Theorem \ref{thm_dist_large} is proved in Section \ref{sec_large}. 

		{\bf Acknowledgement: } We acknowledge Prof. Yi Liu for many helpful discussions, comments suggestions, help and mentoring. We acknowledge Prof. Shicheng Wang for helpful discussion on Remark \ref{rem_conj} and helpful suggestions. We acknowledge Prof. Yunhui Wu and Yang Shen for suggesting me considering the Weil-Petersson sistance version of Theorem \ref{thm_random}
		and acknowledge Prof. Yunhui Wu for many helpful discussion and comments on Theorem \ref{thm_random_wp}. We acknowledge Prof. Jiajun Wang for helpful suggestions. 

 \section{Preliminaries}
 \label{sec_pre}

 \subsection{Teichm\"uller space}

We denote by $\mathcal{T}_g$  the Teichm\"uller space consisting of marked hyperbolic surface with genus $g$, and by $\mathcal{M}_g$ the moduli space consisting of hyperbolic surface with genus $g$. It is known that 
\[
		\mathcal{M}_g \cong \mathcal{T}_g/\Gamma_g.
\] Here $\Gamma_g$ is the mapping class group. 

Teichm\"uller metric is a complete mapping class group equivariant metric on $\mathcal T_g$ defined by dilatation of quasi-conformal map between points in $\mathcal T_g$. For $X,Y \in\mathcal T_g$, distance between $X$ and $Y$ is denoted by $d_{\mathcal T}(X,Y)$. Formal definition of this metric is given in 
Actually in most part of this paper, formal definition of this metric is not needed. 

\subsection{Thurston's metric}
In \cite{thurston1998minimal}, Thurston defined an asymmetric metric on the Teichm\"uller space. 
For $X, Y\in \mathcal{T}_g$ and $f:X\to Y$ a Lipschitz homeomorphism between $X$ and $Y$, we let 
\[
		L(f) = \sup_{x,y\in X x\ne y}\frac{d(f(x),f(y))}{d(x,y)}.
\]
Then this metric is defined as
\[
		d_L(X,Y) = \inf_{f} \left\{ \log L(f)| f:X\to Y\text{ is a Lipschitz homeomorphism} \right\}. 
\]

Thurston has proved that
\begin{theorem}[\cite{thurston1998minimal}]
		For $X,Y\in \mathcal{M}_g$
		\[
				d_{L}(X,Y) = \sup_{\alpha \text{ s.c.c. in }X} \inf_{f:X\to Y \text{Lipschitz homeomorphism} } \log \frac{l_{f(\alpha)(Y)}}{l_{\alpha}(X)}.
		\]
		\label{thm_thurston}
\end{theorem}

For $X,Y\in \mathcal{T}_g$, Rafi and Tao \cite[Equation (2)]{rafi2013diameter} have shown that 
\begin{equation}
		\frac{1}{2}d_L(X,Y)\le d_{\mathcal{T}}(X,Y). 
		\label{for_thurston_metric}
\end{equation}

 \subsection{Topological Morse function and generalized systole}

 \begin{definition}
		 On a topological manifold $M^n$, a function $f:M^n\to \mathbb{R}$ is a topological Morse function if at each point $p\in M$, there is a neighborhood $U$ of $p$ and a map $\psi: U\to \mathbb{R}^n$. Here $\psi$ is a homeomorphism between $U$ and its image, such that $f\circ \psi^{-1}$ is either a linear function or \[
				 f\circ \psi^{-1}((x_1,x_2,\dots,x_n)) = f(p) -x_1^2--x_1^2-\dots-x_j^2+x_{j+1}^2+\dots+x_n^2.
		 \]
		 In the former case, the point $p$ is called a regular point of $f$, while in the latter case the point $p$ is called a singular point with index $j$. 
 \end{definition}

 On a Riemannian manifold $M$, $l_\alpha: M\to \mathbb{R}^+$ is a family of smooth function on $M$ indexed by $\alpha\in I$, called {\em (generalized) length function}. The length function family is required to satisfy the following condition: 
 for every $ p \in M$, there exists a neighborhood $U$ of $p$ and a number $K>0$, such that the set $\{ \alpha| l_\alpha(q)\le K, \forall q\in U\}$ is a non-empty finite set. 
 The {\em (generalized) systole function} is defined as 
 \[
		 \sys(p) := \inf_{\alpha\in I} l_\alpha(p), \forall p\in M.
 \]

 Akrout proved that 
 \begin{theorem}[\cite{akrout2003singularites}]
		 If for any $ \alpha \in I$, the Hessian of $l_\alpha$ is positively definite, then the generalized systole function is a topological Morse function. 
		 \label{thm_akrout}
 \end{theorem}

 The critical point of the systole function is also characterized in \cite{akrout2003singularites}. A $p\in M$ is an eutactic point if and only if it is a critical point of the systole function. 

 We assume that for $ p \in M$, 
 \[
		 S(p) := \left\{ \alpha\in I| l_\alpha(p) = \sys(p) \right\}
 \]

 \begin{definition}
	For $p\in M$, $p$ is eutactic if and only if $0$ is contained in the interior of the convex hull of $\left\{ \mathrm d l_{\alpha}|_p|\alpha\in S(p) \right\}$. 
 \end{definition}

 An equivalent definition is: 
 \begin{definition}
		 $p\in M$ is eutactic if and only for $v\in T_{p}M$, if $\mathrm d l_{\alpha}(v)\ge 0$ for all $\alpha \in S(p)$, then $\mathrm d l_{\alpha}(v)= 0$ for all $\alpha \in S(p)$. 
		 \label{defi_eut}
 \end{definition}

 \subsection{Teichm\"uller space and length function}

For a marked hyperbolic surface $\Sigma$ in the Teichm\"uller space $\mathcal{T}_g$, $\alpha\subset \Sigma$ is a essential simple closed geodesic. Its length is denoted by $l_\alpha(\Sigma)$. In another point of view, $l_\alpha$ is a function on $\mathcal{T}_g$:
\begin{eqnarray*}
		l_\alpha: \mathcal{T}_g &\to & \mathbb{R}^+\\
		\Sigma & \mapsto &l_\alpha(\Sigma).
\end{eqnarray*}

The set of all the shortest geodesics on $\Sigma$ is denoted by $S(\Sigma)$. For $\alpha\in S(\Sigma)$, 
\[
		l_\alpha(\Sigma)\le l_\beta(\Sigma), \forall \text{ simple closed geodesic } \beta\subset \Sigma. 
\]
The length of the shortest geodesics of $\Sigma$ is called {\em systole} of $\Sigma$. 

 Similarly, the systole can be treated as a function on $\mathcal{T}_g$, we denote it as $\sys_g$ or shortly $\sys$. Obviously 
 \[
		 \sys(\Sigma) = l_\alpha(\Sigma) = \inf_{\forall\text{ simple closed geodesic }\beta \subset \Sigma }  l_\beta(\Sigma). 
 \]

 \begin{remark}
		 In a small neighborhood $U$ of $\Sigma$ in $\mathcal{T}_g$, the systole function is realized by the minimum of the lengths of finitely many simple closed geosedsics. 
 \end{remark}

 \begin{remark}
		 Systole function can be also defined as a function on $\mathcal{M}_g$. 
\begin{eqnarray*}
		\sys: \mathcal{M}_g &\to & \mathbb{R}^+\\
		\Sigma & \mapsto &\sys(\Sigma).
\end{eqnarray*}
		 However, the length function $l_\alpha$ is not well-defined on $\mathcal{M}_g$ because of the monodromy.
 \end{remark}

 By \cite{wolpert1987geodesic}, the Hessian of $l_\alpha$ is always positively definite, for any simple closed geodesic $\alpha \subset \Sigma$ with respect to the Weil-Petersson metric. Therefore, we have 
 \begin{cor}[\cite{akrout2003singularites}, Corollary]
		 The systole function is a topological Morse function on $\mathcal{T}_g$.  
 \end{cor}
 The systole function is also a topological Morse function on $\mathcal{M}_g$ because systole function is an invariant function on Teichm\"uller space.

 The set of all the critical points of $\sys_g$ in $\mathcal{T}_g$ is denoted as $Crit(\sys_g)$.

 \subsection{Analytic theory of Teichm\"uller space}
 We recall some definitions and theorems on the analytic theory of Teichm\"uller space. For details, we refer the reader to \cite{imayoshi2012introduction}. 
 \subsubsection{Quasi-conformal mapping on the complex plane}

 \begin{definition}
		 Let $f$ be a orientation-preserving homeomorphism of a domain $D\subset \mathbb{C}$ into $\mathbb{C}$. Then $f$ is quasi-conformal if and only if 
		 \begin{enumerate}
				 \item $f$ is absolutely continuous on lines (ACL). 
				 \item There exists a constant $k$ with $0\le k<1$ such that  
						 \[
								 \left| \frac{f_{\bar{z}}}{f_z}\right| \le k. 
						 \]
		 \end{enumerate}
		 \label{defi_qc}

 \end{definition}

		 \begin{remark}
				 The condition (1) in the above definition implies $f_z$ and $f_{\bar{z}}$ are well-defined on $D$ almost everywhere. 
		 \end{remark}

		 \begin{definition}
				 For a quasi-conformal map $f$ on a domain $D$, the quotient \[
						 \mu(z) = \frac{f_{\bar{z}}(z)}{f_z(z)}
				 \]
				 is called the complex dilatation of $f$ on $D$.
		 \end{definition}

		 We assume $B(\mathbb{C})_1$ is the unit open ball in $L^\infty(\mathbb{C})$, namely 
		 \[
				 B(\mathbb{C})_1 =\left\{ \mu\in L^\infty(\mathbb{C})| \| \mu\|_\infty <1 \right\}. 
		 \]
		 On the complex plane $\mathbb{C}$, for each quasi-conformal map $f$, there exists a function $\mu\in B(\mathbb{C})_1$ such that $f_{\bar{z}} = \mu f_z$. Conversely, for any $\mu\in B(\mathbb{C})_1$ we have 
		 \begin{theorem}[\cite{imayoshi2012introduction}, Theorem 4.30]
				 For any $\mu\in B(\mathbb{C})_1$, there is a homeomorphism $f$ from $\mathbb{C}\cup \left\{ \infty \right\}$ onto $\mathbb{C}\cup \left\{ \infty \right\}$ such that 
				 \begin{enumerate}
						 \item $f$ is a quasi-conformal map on $\mathbb{C}$. 
						 \item $f_{\bar{z}} = \mu f_z$ on $\mathbb{C}$ $a.e.$. 
				 \end{enumerate}
				 Moreover $f$ is unique up to biholomorphic automorphism on $\mathbb{C}\cup \left\{ \infty \right\}$. 
				 \label{thm_qc_exist}
		 \end{theorem}
		 We denote by $f^\mu$ the uniquely determined normalized function satisfies (1) $(f^\mu)_{\bar{z}} = \mu (f^\mu)_z$ and (2) $f^\mu(0) = 0, f^\mu(1) = 1, f^\mu(\infty) = \infty$. The $f^\mu$ is called {\em canonical $\mu$-qc mapping}. 

		 The $f^\mu$ depends on $\mu$ 'continuously' in the following sense: 
		 \begin{proposition}[\cite{imayoshi2012introduction}, Proposition 4.36]
				 If $\mu$ converges to $0$ in $B(\mathbb{C})_1$, then the canonical $\mu$-qc mapping $f^\mu$ converges to the identity mapping locally uniformly on $\mathbb{C}$. 
				 \label{prop_qc_continuous}
		 \end{proposition}

		 For a domain $D\subset \mathbb{C}$, we call any element in $B(D)$ a {\em Beltrami coefficient} on $D$. 

		 For quasi-conformal mappings on the upper half plane $\mathbb{H}^2$, similarly we have 
		 \begin{proposition}[\cite{imayoshi2012introduction}, Proposition 4.33]
				 Let $B(\mathbb{H}^2)_1$ be the unit open ball in $L^\infty(\mathbb{H}^2)$. For any $\mu\in B(\mathbb{H}^2)_1$, there is a quasi-conformal map $w$ from $\mathbb{H}^2$ onto $\mathbb{H}^2$ satisfies 
						  $w_{\bar{z}} = \mu w_z$ on $\mathbb{H}^2$ $a.e.$. 
						  $w$ is unique up to biholomorphic automorphism. 

				 Moreover, $w$ can be extended to a homeomorphism on $\mathbb{H}^2\cup \partial\mathbb{H}^2$. 
				 \label{thm_qc_exist}
		 \end{proposition}

		 We let $w^\mu$ to be the unique quasi-conformal map on $\mathbb{H}^2$, whose Beltrami coefficient is $\mu$, under the normalization condition $w^\mu(0)=0$, $w^\mu(1)=1$ and $w^\mu(\infty)=\infty$. 

		 Similar to $B(\mathbb{C})_1$, $\mu\in B(\mathbb{H}^2)_1$ converges to $0$ implies $w^\mu$ converges to the identity mappings locally uniformly on $\mathbb{H}^2$.

 \subsubsection{Quasi-conformal mapping on Riemann surfaces}

 \begin{definition}
		 A homeomorphism $f:R\to S$ between two closed genus $g$ ($g\ge 2$) Riemann surface $R$, $S$ is a quasi-conformal map if 
		 for the universal cover of $R$ and $S$, denoted $\tilde{R}$ and $\tilde{S}$ respectively, 
		 there is a lift $\tilde{f}: \tilde{R} \to \tilde{S}$ of $f$, such that $\tilde{f}$ is a quasi-conformal map as is defined in Definition \ref{defi_qc}. 
 \end{definition}
 We remark that this definition is independent to the choice of the lift. 

 We need an analytic definition of Teichm\"uller space with a base point:
 \begin{definition}
		 For a closed Riemann surface $R$ with genus at least $2$, we consider a pair $(S,f)$, here $S$ is also a Riemann surface and $f:R\to S$ is a quasi-conformal map from $R$ onto $S$. Two pairs $(S_1,f_1)$, $(S_2,f_2)$ are equivalent if $f_2\circ f_1^{-1}$ is homotopic to a conformal map. The equivalent class is denoted by $[S,f]$. 

		 The Teichm\"uller space with a base point $R$ (denoted $T(R)$) consists of the equivalent classes $[S,f]$. 
 \end{definition}

 Similarly the Teichm\"uller space of Fuchsian group with a base point is defined as follows: 

 \begin{definition}
		 For a Fuchsian group $\Gamma$, we let $w$ be a quasiconformal map on $\mathbb{C}\cup \left\{ \infty \right\}$ such that \begin{enumerate}
				 \item $w\Gamma w^{-1}$ is also a Fuchsian group.
				 \item $w(0)=0$, $w(1)=1$ and $w(\infty)=\infty$. 
		 \end{enumerate}

		 Two such maps $w_1$, $w_2$ are equivalent if $w_1=w_2$ on the real axis. 

		 The Teichm\"uller space of Fuchsian group with a base point $\Gamma$ (denoted $T(\Gamma)$) consists of the equivalent classes $[w]$. 
 \end{definition}
 If $R$ is a closed Riemann surface and $\Gamma$ is a Fuchsian model of $R$, then $T(\Gamma)$ is identified with $T(R)$. 

 For any $\mu\in B(\mathbb{H}^2)$, $w^\mu\Gamma (w^\mu)^{-1}$ is Fuchsian if and only if for any $ \gamma \in \Gamma$
 \begin{equation}
		 \mu = (\mu\circ \gamma )\frac{\overline{\gamma'}}{\gamma'} \text{\,\,a.e. on } \mathbb{H}^2. 
		 \label{for_fuchsian}
 \end{equation}

 We let 
 \[
		 B(\mathbb{H}^2,\Gamma) = \left\{ \mu\in L^{\infty}|
		 \mu \text{ satisfies (\ref{for_fuchsian}), } \forall \gamma \in \Gamma\right\}. 
 \]

 For $\mu\in B(\mathbb{H}^2,\Gamma)$, $\mu$ is called a Beltrami differential on $\mathbb{H}^2$ with respect to $\Gamma$. 
 Also, we set
 \[
		 B(\mathbb{H}^2,\Gamma)_1 = \left\{ \mu\in B(\mathbb{H}^2,\Gamma) |\|\mu\|_\infty <1\right\}.
 \]
An element in $B(\mathbb{H}^2,\Gamma)_1$ is called a Beltrami coefficient on $\mathbb{H}^2$ with respect to $\Gamma$.

There is a projection $H$ from $B(\mathbb{H}^2,\Gamma)$ to itself.  
\[
		H:B(\mathbb{H}^2,\Gamma) \to B(\mathbb{H}^2,\Gamma). 
\]
The image of $H$ is denoted $HB(\mathbb{H}^2,\Gamma)$.
An element in $HB(\mathbb{H}^2,\Gamma)$ is called a {\em harmonic Beltrami differential}. The tangent space of $T(\Gamma)$ at the base point $\Gamma$ is identified with $HB(\mathbb{H}^2,\Gamma)$. 

To obtain Lemma \ref{lem_harmonic}, we briefly recall how to construct $H[\mu]$ from $\mu$. 
\begin{equation}
		H[\mu](z) = -\frac{2}{(\mathrm{Im}\: z)^2}\dot \Phi_0[\mu](\bar{z}). 
		\label{for_harmonic}
\end{equation}
Here 
\[
		\dot \Phi_0[\mu](\bar{z}) = -\frac{6}{\pi}\iint_{\mathbb{H}^2} \frac{\mu(\zeta)}{(\bar{z}-\zeta)^4} \mathrm{d} \xi \eta
\] 
is the derivative of the Bers' embedding at the base point and $\zeta = \xi+i\eta$. 

On the tangent space $T(\Gamma) = HB(\mathbb H^2,\Gamma)$, the real $L^2$-inner product of the harmonic Beltrami differential is called the {\em (real) Weil-Petersson inner product} and denoted by $(,)_{wp}$. This inner product defines a Riemannian metric on the Teichm\"uller space, equivariant to the mapping class group, called the {\em Weil-Petersson metric}. Distance between two points in $\mathcal T_g$ or $\mathcal M_g$ with respect to Weil-Petersson is called {\em Weil-Petersson distance} and denoted by $d_{wp}(,)$.

\subsection{Quadratic differential and Teichm\"uller geodesics}

This subsection is a short sketch on Teichm\"uller geodesics and extremal length, for details, see \cite{masur2009geometry}.

For a quasi-conformal map $f:X\to Y$, $X,Y\in \mathcal{T}_g$. 
We let \[
		K(f) = \sup_{z\in X} \frac{1+ |\mu_f(z)|}{1- |\mu_f(z)|}. 
\]
Here $\mu_f$ is the complex dilatation of $f$ defined in the last subsection.

The Teichm\"uller distance on $\mathcal{T}_g$ is defined to be 
\[
		d_{\mathcal{T}}(X) = \frac{1}{2}\inf_{f\sim id} \left\{ \log K(f)| f:X\to Y \right\}. 
\]

A Teichm\"uller geodesic ray with respect to Teichm\"uller distance from $X\in \mathcal{T}_g$ can be induced from a holomorphic quadratic differential $q$ on $X$. A {\emph holomorphic quadratic differential} is a tensor locally written as $\psi(z)\mathrm{d}z^2$, where $\psi(z)$ is a holomorphic function. We denote the space of quadratic differentials on $X$ as $QD(X)$. The bundle of quadratic differentials over $\mathcal{T}_g$ is denoted by $QD_g$.

For $X\in \mathcal T_g$ and $q\in QD(X)$, for any $0<k<1$, $\mu_k=k\frac{\bar{q}}{q}$ is a Beltrami coefficient on $X$. We let $f_k$ be the quasi-conformal induce by $\mu_k$, $f_k:X\to X^{(k)}$. Then $f_k$ is the Teichm\"uller map from $X$ to $X^{(k)}$ and the Teichm\"uller geodesic ray induced by $(X,q)$ consists of all the $X^{(k)}$ for all $k\in (0,1)$.

A non-zero $q\in QD(X)$ has a canonical coordinate. In this coordinate, $q$ can be locally written as $\mathrm{d} z^2$ in the neighborhood of any non-zero point of $q$ and $q$ has only finitely many zero points. 

The quadratic differential $q$ determines a pair of tranversed measured foliation on $X$, called {\emph horizontal and vertical foliations} for $q$ and denoted by $F_h(q)$ and $F_v(q)$ respectively. In the canonical coordinate of $q$, the leaves of $F_h(q)$ is given by $y=\text{const}$ and the leaves of $F_v(q)$ is given by $x=\text{const}$. Here $z = x + iy$ is the coordinate. The measure of  $F_h(q)$ and  $F_v(q)$ are given by $|\mathrm d y|$ and $|\mathrm d x|$ respectively. 

For $X_t$ on the geodesic induced by $(X,q)$ with $d_{\mathcal T}(X,X_t)=t$, there is a quadratic differential $q_t\in QD(X_t)$ as the push forward of $q$ by $f_t$. We let $z=x+iy$ be the canonical coordinate of $(X,q)$ and $w=u+iv$ be the canonical coordinate of $(X,q)$, then 

\begin{equation}
		u = e^tx \text{ and } v= e^{-t}y
		\label{for_teich_dist}
\end{equation}

\subsection{Jenkins-Strebel differential}

A quadratic differential $q\in QD(X)$ is called a {\emph Jenkins-Strebel differential} if 
any leaf of $F_h(q)$ and $F_v(q)$ is a simple closed curve except finitely many leaves that connect zeros of $q$. 

For a Jenkins-Strebel differential $q\in QD(X)$ and a simple closed leaf $\alpha$ of $F_h(q)$, all simple closed leaves of $F_h(q)$ parallel to $\alpha$ form a cylinder in $X$. This cylinder is called the {\em characteristic ring domain} of $\alpha$. This cylinder, with respect to the metric $|q|$ is an Euclidean cylinder, isometric to $R=[0,a]\times (0,b)/((0,t)\sim (a,t),0<t<b)$. We call $a$ the \emph length of $R$ and $b$ the \emph height of $R$.

We need the following theorem on Jenkins-Strebel differential: 

\begin{theorem}[\cite{strebel1984quadratic}, Theorem 21.1 ]
		Let $(\gamma_1,\dots,\gamma_p)$ be a finite, pairwisely disjoint, essential curve system in $X\in \mathcal{T}_g$.  For each $\gamma_i$, there is a regular neighborhood $R'_i$ of $\gamma_i$ in $X$ and $R'_1,\dots,R'_p$ are pairwisely disjoint. Then for any $(b_1,\dots,b_p) \in \mathbb{R}^p_+$, there is a unique Jenkins-Strebel differential $q\in QD(X)$ such that 
		\begin{itemize}
				\item $\gamma_i$ is a leaf of $F_h(q)$. 
				\item A simple closed leaf of $F_h(q)$ must be freely homotopic to a $\gamma_i$, here $i=1,2,\dots,\text{ or }, p$. 
				\item Let $R_i$ be the characteristic ring domain of $\gamma_i$, then the height of $R_i$ is $b_i$ with respect to the metric $|q|$. 
		\end{itemize}
		\label{thm_bi}
\end{theorem}

\subsection{Extremal length and hyperbolic length}

The definition of {\emph extremal length} of an essential curve $\alpha$ in a Riemann surface $X$ is given by
\[
		\Ext_\alpha(X) = \sup_{\rho} \frac{l_{\alpha}(\rho)^2}{{\mathrm Area}(X,\rho)}. 
\]
Here the supremum is taken over all metrics $\rho$ conformal to the metric on $X$, $l_{\alpha}(\rho)$ is the length of $\alpha$ in the metric $\rho$ and ${\mathrm Area}(X,\rho)$ is the area of $X$ in the metric $\rho$. 

For a Euclidean cylinder with length $a$ and height $b$, the extremal length of its core curve in the cylinder is $a/b$, see for example \cite{ahlfors2006lectures}. 

We can express distance between points on a Teichm\"uller geodesic by extremal lengths of  horizontal foliation leaves in their characteristic ring domains. 
For a Jenkins-Strebel differential $q\in QD(X)$, we let $\alpha$ be a simple closed leaf of $F_h(q)$ and $R$ be the characteristic ring domain of $\alpha$ with length $a$ and height $b$. We consider $X_t$ on the Teichm\"uller geodesic induced by $(X,q)$ with $d_{\mathcal T}(X,X_t)=t$. 
We let $q_t\in QD(X_t)$ be the push forward of $q$ by the Teichm\"uller map $f_t:X\to X_t$. 
Hence by (\ref{for_teich_dist}) and definition of horizontal foliation, 
$f_t(\alpha)$ is a leaf of $F_h(q_t)$, and
$f_t(R)$ is the characteristic ring domain of $f_t(\alpha)$ with length $e^ta$ and height $e^{-t}b$. Then 
$\Ext_{\alpha}(R)= \frac{a}{b}$ and $\Ext_{f_t(\alpha)}(f_t(R))= e^{2t}\frac{a}{b}$. 
Then 
		\begin{equation}
				d_{\mathcal{T}}(X,X_t) = \frac{1}{2}\left |\log \frac{\Ext_{f_t(\alpha)}(f_t(R))}{\Ext_{\alpha}(R)} \right |. 
				\label{for_dist_extremal}
		\end{equation}

The rest of this subsection is the comparison between hyperbolic length and extremal length by Maskit.

		For a simple closed geodesic $\alpha$ in a hyperbolic surface $X$, the \em collar of $\alpha$ with width $w$ is an embeded cylinder in $X$ consisting of points with distance at most $w$ to $\alpha$. 
		\begin{theorem}[\cite{maskit1985comparison}] In hyperbolic surface $X$, if a simple closed geodesic $\alpha$ has collar $C$ with width $\arccosh \frac{1}{\cos \theta}$ then 
				\[
						\frac{1}{\pi}l_{\alpha}(X) \le \Ext_{\alpha}(X)\le \Ext_{\alpha}(C)\le \frac{1}{2\theta}l_{\alpha}(X).
				\]
				\label{thm_maskit}
		\end{theorem}

\subsection{The Fenchel Nielsen deformation}

The Fenchel Nielsen deformation at time $t$ along a simple closed geodesic is defined as follows: 
For $X\in \mathcal{T}_g$ and $\alpha\subset X$ is a simple closed curve, we cut $X$ along $\alpha$ and get a surface (could be disconnected) with boundary. A new surface $X_t$ is constructed by gluing the boundary components back with a left twist of distance $t$ (see Figure \ref{fig_twist_1}, \ref{fig_twist_2}). The term left twist means the image of any oriented curve $\beta\subset X$ transverse to $\alpha$ in $X_t$ 'turns left' with respect to $\beta$'s orientation when meeting $\alpha$. 

\begin{figure}
		\begin{subfigure}{.45\linewidth}
				\centering
				\includegraphics{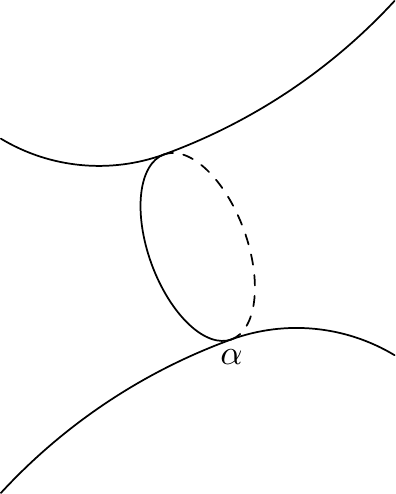}
				\caption{Neighborhood of $\alpha$ in $X$ }
				\label{fig_twist_1}
		\end{subfigure}
		\begin{subfigure}{.45\linewidth}
				\centering
				\includegraphics{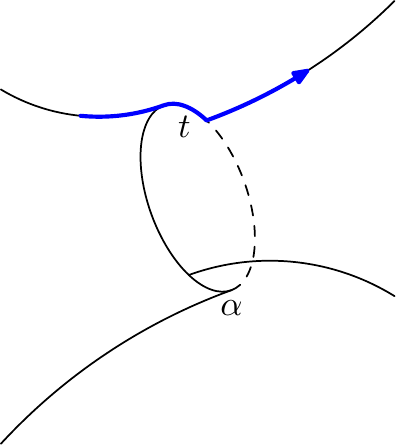}
				\caption{Neighborhood of $\alpha$ in $X_t$ }
				\label{fig_twist_2}
		\end{subfigure}

		\caption{The Fenchel-Nielsen deformation on a simple closed curve}
		\label{fig_twist}
\end{figure}

The Fenchel-Nielsen definition along a family of of disjoint simple closed geodesics is defined as the simutaneous Fenchel-Nielsen deformation along every curve in the family. 

\subsection{Hyperbolic trigonometric formulae \cite{buser2010geometry}}

 {Right angle triangle (Figure \ref{fig_right_angle_triangle})  \cite[p.454]{buser2010geometry}} 
 \begin{equation}
		 \cosh c = \cot \alpha \cot \beta. \label{for_tri_2}
 \end{equation}

Trirectangles (Figure \ref{fig_trirectangle}) \cite[p.454]{buser2010geometry} 
\begin{equation}
		\cos \varphi = \sinh a \sinh b. \label{for_trirect_1} 
\end{equation}

Right-angled pentagon (Figure \ref{fig_pentagon})  \cite[p.454]{buser2010geometry}
\begin{equation}
		\cosh c = \sinh a \sinh b.
		\label{for_pentagon}
\end{equation}

Right-angled hexagon (Figure \ref{fig_hexagon}) \cite[p.454]{buser2010geometry}
\begin{equation}
		\cosh c = \sinh a \sinh b \cosh \gamma -\cosh a \cosh b.
		\label{for_hexagon}
\end{equation}
\begin{figure}
\begin{subfigure}{0.45\linewidth}
\centering
\includegraphics{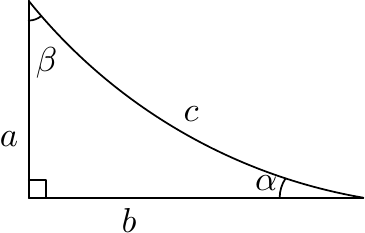}
\caption{}
\label{fig_right_angle_triangle}
\end{subfigure}
		\begin{subfigure}{.45\linewidth}
				\centering
				\includegraphics{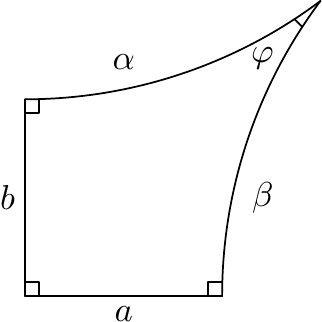}
				\caption{The trirectangle}
				\label{fig_trirectangle}
		\end{subfigure}

		\begin{subfigure}{.45\linewidth}
				\centering
				\includegraphics{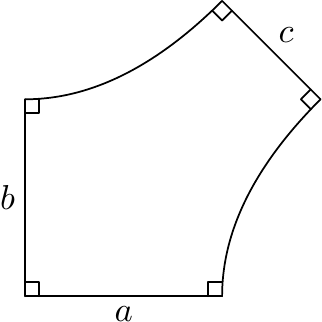}
				\caption{The right-angled pentagon}
				\label{fig_pentagon}
		\end{subfigure}
		\begin{subfigure}{.45\linewidth}
				\centering
				\includegraphics{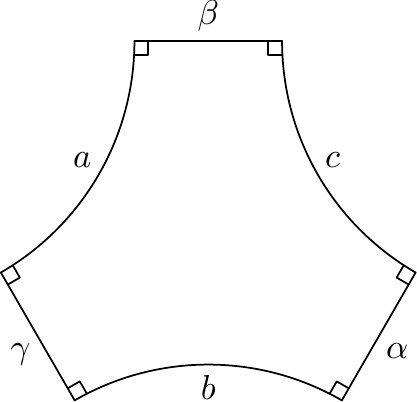}
				\caption{The right-angled hexagon}
				\label{fig_hexagon}
		\end{subfigure}
		\caption{Hyperbolic polygons}
		\label{fig_hyp_polygon}
\end{figure}

 \section{The surface $S_g$}
 \label{sec_hole}

In this section we construct a surface $S_g$, whose Teichm\"uller distance to $X_g$ is at least $\frac{1}{4} \log\left( \log{g} - \log 12 \right)$. 

\subsection{Construction of the surface $S_g$ when $g=3\cdot 2^{n-1}$}

To construct a surface $S_g$, we first construct a trivalent tree $T(n)$ with $m$ vertices. Diameter of the tree is required to be  comparable with $\log m$. 

We pick three full binary trees of depth $n$. Hence each of the tree has $2^{n}-1$ vertices and $2^{n-1}$ leaves. 

We join the trees at their roots by a vertex $O$ , then we get a trivalent tree with $3(2^n-1)+1$ vertices and $3\cdot 2^{n-1}$ leaves (see Figure \ref{fig_pants_2}). Distance from $O$ to any leaf is at least $n$. We denote the tree by $T(m)$ where $m=3\cdot (2^n-1)+1$ is the number of vertices in the tree. 
 \begin{figure}[htbp]
		 \centering
		 \includegraphics{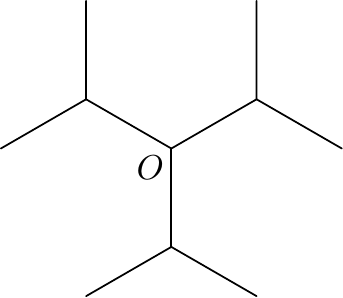}
		 \caption{The tree}
		 \label{fig_pants_2}
 \end{figure}

 Now we construct the surface $S_g$ from the tree $T(m)$. 
 We pick several isometric pairs of pants. Each pair consists of two regular right-angled hexagons. A boundary component of the pants is called a \em cuff and an edge of the hexagons in the interior of the pants is called a a seam. We glue the pants together according to the trivalent tree $T(m)$. A vertex of $T(m)$ corresponds to a pair of pants, two pairs of pants are glued togeter at a cuff if the corresponding vertices are connected by an edge. Now we get a sphere with $6\cdot 2^{n}$ boundary components(Figure \ref{fig_pants_1}). For each pair of pants corresponding to a leaf in the tree, we glue  together the two cuffs of the pair that are not glued with other pants. Then we get a closed surface. At each cuff, ends of seams from the two sides of the cuff are required to be glued together. This surface we construct is the surface $S_g$, where $g=3\cdot 2^{n-1}$ equals the number of leaves of $T(m)$. The construction of $S_g$ for $3\cdot 2^{n-1}<g<3\cdot 2^{n}$ is delayed in the end of this section. 

 \begin{figure}[htbp]
		 \centering
		 \includegraphics{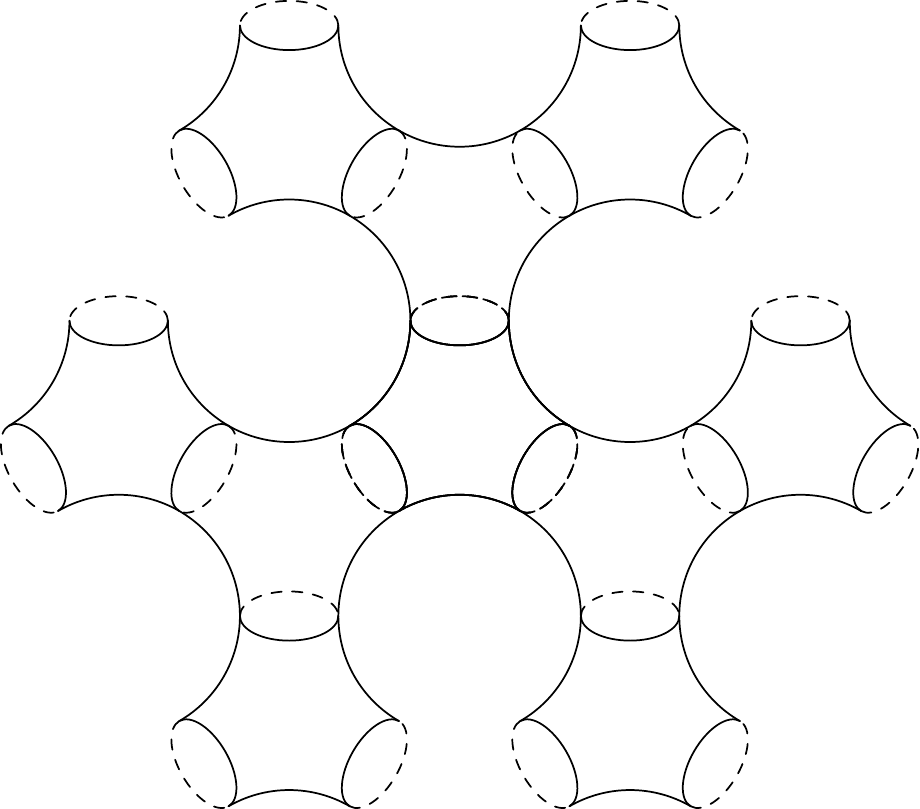}
		 \caption{The sphere with $6\cdot 2^{n}$ boundary components}
		 \label{fig_pants_1}
 \end{figure}

 In $S_g$, in each one-holed torus (glued from a pair of pants) corresponds to a leaf of the tree, there is a unique simple closed curve consists of one seam of the pants. We denote this curve by $\alpha_k$ where $k=1,2,\dots,g$. Now we prove that this curve is the shortest curve in $S_g$. 

 \begin{lemma}
		 The shortest geodesics in $S_g$ consists of $\alpha_k$, $k=1,2,\dots,g$ and therefore systole of $S_g$ is $\arccosh 2$. 
		 \label{lem_sys_r}
 \end{lemma}
 \begin{proof}

		 By definition, the length of $\alpha_k$ is equal to the edge length of the regular right-angled hexagon and hence equal to the distance between two cuffs in the pants consisting of regular right-angled hexagons. 
		 Therefore by (\ref{for_hexagon}), length of $\alpha_k$ is $\arccosh 2$. 

		 Then if a curve in $S_g$ intersects at least three pairs of pants, then this curve is longer than $\alpha_k$ because this curve must pass two cuffs that belongs to one of the three pants. 

		 In a pair of pants consists of two regular right-angled hexagons, the only simple closed geodesics are the cuffs. The cuff length of the pants is exactly twice the length of $\alpha_k$. 

		 If a curve is contained in two neighboring pairs of pants, then it intersects the shared cuff of the two pants and the seam in each pants opposite to the cuff. However, by (\ref{for_pentagon}), distance between the cuff and the seam are larger than the length of $\alpha_k$. 

		 Therefore $\left\{ \alpha_k \right\}_{k=1}^{g}$ is the set of shortest geodesics of $R(n)$.
 \end{proof}

 \begin{remark}
		 By \cite[Corollary 20]{schaller1999systoles}, $S_g$is a regular point of the systole function since its shortest geodesics do not fill.
 \end{remark}

 \subsection{Distance between $S_g$ and $X_g$}
The distance between a surface and $X_g$ is estimated from below by the following lemma:
 \begin{lemma}
		 For a surface $S\in \mathcal M_g$, 
		 let $L>0$. 
		 If for any filling curve set $F$ in which each pair of curves intersect at most once, $F$ contains a curve longer than $L$, then 
		 \begin{equation}
				 d_{\mathcal T}(S,X_g)\ge \frac{1}{4}\log\frac{L}{\sys(S)}.
				 \label{for_dist_x}
		 \end{equation}
		 \label{lem_dist_x}
 \end{lemma}
 \begin{proof}
		 We let $S\in \mathcal M_g$ is a surface. For any filling curve set $F \subset S$, in which each pair of curves intersect at most once, $F$ contains a curve longer than $L$. 

		 For any $S'\in X_g$, we assume $F'\subset S'$ is the set of shortest geodesics in $S'$. Since $S'\in X_g$, $F'$ is filling in $S'$. 

		 For any Lipschitz homeomorphism $f:S\to S'$ and $g:S'\to S$, we let $\alpha\subset S$ be a shortest geodesic in $S$ and $\beta\subset S'$ be a shortest geodesic with $l_{g(\beta)(S)}>L$. Then by Theorem \ref{thm_thurston}
		 \[
				 \exp(d_L(S,S'))\ge \frac{l_{f(\alpha)}(S')}{l_{\alpha}(S)} \ge \frac{\sys(S')}{\sys(S)}. 
		 \]
		 On the other hand, 
		 \[
				 \exp(d_L(S',S))\ge \frac{l_{g(\beta)}(S)}{l_{\beta}(S')}\ge \frac{L}{\sys(S')}. 
		 \]
		 Then by (\ref{for_thurston_metric}), $d_{\mathcal T}(S,S')\ge \frac{1}{2}d_{ L}(S,S')$ and $d_{\mathcal T}(S,S')\ge \frac{1}{2}d_{ L}(S',S)$. 
		 For any $\sys(S')>0$, 
		 \[
				 \max\left( \frac{\sys(S')}{\sys(S)},\frac{L}{\sys(S')} \right) \ge \sqrt{\frac{L}{\sys(S)}}.
		 \]
		 Therefore 
		 \[
				 d_{\mathcal T}(S,S')\ge \frac{1}{2}\log \sqrt{\frac{L}{\sys(S)}} = \frac{1}{4}\log \frac{L}{\sys(S)}. 
		 \]

 \end{proof}

 Now we begin to estimate the distance between $S_g$ and $X_g$ using Lemma \ref{lem_dist_x}. 

 We let $P_k$, $k=1,\dots,g$ be the one-holed tori corresponds to leaves of the tree $T(m)$. An observation is that $S_g\backslash \left\{ P_k \right\}_{k=1}^{g}$ is a $g$-holed sphere. 

 Immediately we have
 \begin{lemma}
		 In $S_g$, for any filling curve set $F$ in which each pair of curves intersect at most once, any curve in $F$ intersects at least one $P_k$ in $\left\{ P_k \right\}_{k=1}^{g}$. 
		 \label{lem_separate}
 \end{lemma}
 \begin{proof}
		 If a curve does not intersect any $P_k$ for $k=1,2,\dots,g$, then it is contained in the $g$-holed sphere $S_g\backslash \left\{ P_k \right\}_{k=1}^{g}$, hence is a separating curve. A separating curve cannot intersect any curve once. On the other hand, a curve in a filling set $F$ always intersects other curves in $F$. Therefore this lemma holds. 
 \end{proof}

 \begin{lemma}
		 In $S_g$, for any filling curve set $F$ in which each pair of curves intersect at most once, $F$ contains a curve $\beta$, 
		 \[
				 l_{\beta}(S_g)>{n}\arccosh 2,
		 \]
		 \label{lem_length_beta}
		 where $g=3\cdot2^{n-1}$. 
 \end{lemma}
 \begin{proof}
		 The construction of $S_g$ gives a natural pants decomposition on $S_g$. A filling curve set must intersect every pair of pants in this decomposition because filling curve sets cut the surface into disks. 

		 For the pants corresponds to the center vertex $O$ shown in Figure \ref{fig_pants_2}, we let $\beta$ be a curve in $F$ passing this pair of pants. 
		 Then by Lemma \ref{lem_separate}, $\beta$ intersectssome one-holed sphere corresponding to a leaf in the tree $T(m)$. 
		 Distance between the vertex $O$ and any leaf of the tree is at least $n$. Then by the construction of $S_g$, distance between the corresponding two pairs of pants is at least $n\arccosh 2$, where $\arccosh 2$ is the length of seams of the pants constructing $S_g$. 

		 Therefore
				 $l_{\beta}(S_g)>{n}\arccosh 2$.
 \end{proof}

 By Lemma \ref{lem_length_beta} and Lemma \ref{lem_dist_x} immediately we have
 \begin{proposition}
		 When $g= 3\cdot 2^{n-1}$ for any positive integer $n$, distance between $S_g$ and $X_g$ is larger than
		 \[
				 d_{\mathcal T}(S_g, X_g) > \frac{1}{4}\log n. 
		 \]
		 \label{prop_hole}
 \end{proposition}
 \subsection{Construction in general genus}
 We have proved Theorem \ref{thm_hole} when $g= 3\cdot 2^{n-1}$. Now we construct $S_g$ when $3\cdot 2^{n-1}<g<3\cdot 2^{n}$. 

 When $3\cdot 2^{n-1}<g<3\cdot 2^{n}$, we take a tree $T$ with $g$ leaves, such that $T(3(2^{n}-1)+1)\subset T\subset T(3(2^{n+1}-1)+1)$. 
 By the embedding $T(3(2^{n}-1)+1)\to T$, we define the vertex of $O$ in $T$ to be the image of vertex $O$ in $T(3(2^{n}-1)+1)$. 
 Then in the tree $T$ distance from $O$ to any leaf of $T$ is larger than $n$. 

 Similarly to the construction in the beginning of this section we can construct a genus $g$ surface $S_g$ from the tree $T$. By Lemma \ref{lem_dist_x}, distance between $S_g$ and $X_g$ is larger than $\frac{1}{4}\log n$. Since $g< 3\cdot 2^n$, we have 
 \begin{theorem}
		 For any $g\ge 3$, the distance from the surface $S_g$ to the space $X_g$ is larger than 
		 \[
				 d_{\mathcal{T}}(S_g,X_g)> \frac{1}{4} \log\left( \log\frac{g} - \log 12 \right).
		 \]
		 \label{thm_hole}

 \end{theorem}

 \section{Sparseness of $X_g$}
 \label{sec_random}
 \subsection{Two theorems on random surface}

 We list two theorems of random surfaces we need for the proof of Theorem \ref{thm_random}. 

 \begin{theorem}[\cite{mirzakhani2019lengths}, Theorem 2.8]
		 There exists $A,B>0$ so that for any sequence $\{c_g\}$ of positive numbers with $c_g< A\log g$, we have
		 \[
				 P_{WP}( S\in \mathcal M_g| \sys(S)>c_g ) < Bc_ge^{-c_g}. 
		 \]
		 \label{thm_mp}
 \end{theorem}
 In a hyperbolic surface, the {\em half collar} of a simple closed geodesic $\gamma$ with width $w$ is an embeded cylinder in the surface. One of the boundary curves of the cylinder is the geodesic $\gamma$ and this cylinder consists of points with distance at most $w$ to $\gamma$ in one side of $\gamma$. 
 \begin{theorem}[\cite{nie2020large} Theorem 1 and Theorem 2]
		 For any $\varepsilon>0$, consider the following condition:

		 (a) There is a simple closed curve $\gamma$ in $S$, it has a half collar with width $\frac{1}{2}\log g-(\frac{3}{2}+\varepsilon)\log\log g$. 

		 (b) Length of the curve $\gamma$ in (a) is larger than $2\log g -5 \log \log g$. 

		 Tehn we have 
		 \[
				 P_{WP}(S\in \mathcal M_g| S\text{ satisfies (a) and (b)}) \to 1
		 \]
		 as $g\to \infty$. 
		 \label{thm_nwx}
 \end{theorem}

 \subsection{The sparseness of $X_g$}

 Now we are ready to prove Theorem \ref{thm_random}. 

\begin{theorem}
		For the probability of a point in $\mathcal M_g$ whose Teichm\"uller distance to $X_g$ is smaller than $\frac{1}{5}\log \log g$  
		\[
		P_{WP}(S|d_{\mathcal T}(S,X_g)<\frac{1}{5}\log\log g ) \to 0		\]
		as $g\to \infty$. 

		\label{thm_random}
\end{theorem}

\begin{proof}
		By Theorem \ref{thm_mp}, we let $c_g= \frac{1}{5}\log\log g$, then 
		 \[
				 P_{WP}( S\in \mathcal M_g| \sys(S)>\frac{1}{5}\log\log g ) < B\frac{\frac{1}{5}\log \log g}{(\log g)^{\frac{1}{5}}}. 
		 \]

		 For $S\in \mathcal M_g$ and $\sys(S)\le \frac{1}{5}\log\log g$, if $S$ satisfies the condition (a) in Theorem \ref{thm_nwx}, then for any filling curve set $F$ in $S$, $F$ contains a curve of length at least $\log g-2\log\log g$ since in $F$ there must be a curve intersect the separating curve $\gamma$ in condition (a). 
		 Then by Lemma \ref{lem_dist_x}, distance between $S$ and $X_g$ is bounded below by
		 \[
				 \frac{1}{4} \log \frac{\log g-2\log\log g}{\frac{1}{5}\log\log g}> \frac{1}{5}\log\log g. 
		 \]
		 By Theorem \ref{thm_nwx}, 
		 $P_{WP}(S\in \mathcal M_g| d_{\mathcal T}(S,X_g)>\frac{1}{5}\log\log g) \to 1$ as $g\to \infty$ and the theorem holds. 
\end{proof}

Recall that $X_g$ is contained in the thick part $\mathcal M_g^{\ge\varepsilon}$ in $\mathcal M_g$. 
The thick part  $\mathcal M_g^{\ge\varepsilon}$ has positive probability in $\mathcal M_g$ by \cite[Theorem 4.1]{mirzakhani2019lengths}, immediately we have 

\begin{cor}
		When $g\to \infty$, 
		\[
				\frac{P_{WP}(S|d_{\mathcal T}(S,X_g)<\frac{1}{5}\log\log g ) }{P_{WP}(S\in \mathcal M_g^{\ge\varepsilon})}\to 0. 
\]


		\label{cor_thick}
\end{cor}

\section{The Weil-Petersson distance version of Theorem \ref{thm_random}}
\label{sec_random_wp}
Besides the Teichm\"uller distance, if we consider the Weil-Petersson distance to $X_g$, we can prove Theorem \ref{thm_random_wp}. 

\subsection{Lower bounds on Weil-Petersson distance}

The main tools to prove Theorem \ref{thm_random_wp} are Theorem \ref{thm_mp}, Theorem \ref{thm_nwx} and the lower bounds on Weil-Petersson distance by Wu in \cite{wu2020new}. 

 Before stating Wu's result, We prepare soem definitions, for details, see \cite{wu2020new}.  

 We let $\mathcal M_{-1}$ be the space of complete Riemannian metric on topological surface $S_g$ whth constant curvature $-1$. 
 Then by definition of Teichm\"uller space, $\mathcal T_g = \mathcal M_{-1}/\mathrm{Diff}_0(S_g)$ where $\mathrm{Diff}_0(S_g)$ is the group of diffeomorphism of $S_g$ isotopic to the identity. Let $\pi:\mathcal M_{-1}\to \mathcal T_g$ be the natural projection. 
 We call a smooth path $c(t)\subset \mathcal M_{-1}$ a {\em horizontal curve} if there exists a holomorphic quadratic differential $q(t)$ on $c(t)$ such that $\frac{\partial c(t)}{\partial t} = q(t)$. 

 On a surface $X\in \mathcal M_{-1}$ for $p\in X$, we let $\inj_X(p)$ be the {\em injective radius} of $X$ at $p$, namely the half length of shortest essential loop on $X$ passing $p$. Then we define
 \begin{definition}
		 On a topological surface $\Sigma_g(g\ge2)$, fix $p\in \Sigma_g$. For any $X,Y\in \mathcal T_g$, we define
		 \[
				 \left|\sqrt{\inj_X(p)}-\sqrt{\inj_Y(p)}\right|:=\sup_{c}\left|  \sqrt{\inj_{c(0)}(p)}-\sqrt{\inj_{c(1)}(p)}\right|
		 \]
		 where $c:[0,1]\to \mathcal M_{-1}$ runs over all smooth horizontal curves wth $\pi(c(0))=X$, $\pi(c(1))=Y$ and $\pi(c([0,1))\subset \mathcal T_g$ is the Weil-Petersson geodesic connecting $X$ and $Y$. 
 \end{definition}

 \begin{theorem}[\cite{wu2020new} Theorem 1.1]
		 For a topological surface $\Sigma_g$ with $g\ge2$, fix a point $p\in S_g$. Then for any $X,Y\in \mathcal T_g$ 
		 \[
				 \left|\sqrt{\inj_X(p)}-\sqrt{\inj_Y(p)}\right|\le 0.3884 d_{wp}(X,Y)
		 \]
		 where $d_{wp}(X,Y)$ is the Weil-Petersson distance. 
		 
		 \label{thm_inj}
 \end{theorem}

 A corollary to this theorem is also needed

 \begin{cor}[\cite{wu2020new} Corollary 1.2]
		 For $X,Y\in \mathcal T_g$, 
		 \[
				 \left| \sqrt{\sys(X)}-\sqrt{\sys(Y)} \right| \le 0.5492\, d_{wp}(X,Y)
		 \]
		 \label{cor_sys_dist_wp}
 \end{cor}
 \begin{remark}
		 Before this corollary, the function $\sqrt{\sys(\cdot)}$ was proved to be uniformly Lipschitz on $\mathcal T_g$ endowed with the Weil Petersson metric by Wu in \cite{wu2019growth}. 
 \end{remark}
 \subsection{The theorem with respect to Weil-Petersson distance}

 Now we begin to prove Theorem \ref{thm_random_wp}. 
 First we prove the following two lemmas
 \begin{lemma}
		 If $S\in \mathcal T_g$, satisfying the conditions (a) and (b) in Theorem \ref{thm_nwx}, then there is a curve $\alpha\subset S$, freely homotopic to the geodesic $\gamma$ in the conditions (a) and (b), for any point $p\in \alpha$, 
		 \[
				 \inj_S(p)\ge \frac{1}{4}\log g-(\frac{3}{4}+\frac{\varepsilon}{2})\log\log g. 
		 \]
		 \label{lem_inj_nwx}
 \end{lemma}
 \begin{proof}
		 By condition (a) and (b), $\gamma\subset S$ is a simple closed geodesic of length $2\log g -5 \log \log g$, having a half collar of width $\frac{1}{2}\log g-(\frac{3}{2}+\varepsilon)\log\log g$. Then let $\alpha$ be the curve in the halfcollar of $\gamma$, consisting of points whose distance to $\gamma$ is $\frac{1}{4}\log g-(\frac{3}{4}+\frac{\varepsilon}{2})\log\log g$. This lemma follows immediately. 
 \end{proof}

 \begin{lemma}
		 For any surface $S'\in X_g$, on any essential curve $\alpha'\subset S'$, there is at least one point $p'\in \alpha'$, such that \[
				 \inj_{S'}(p')\le \frac{1}{2}\sys(S')
		 \]
		 \label{lem_inj_fill}
 \end{lemma}
 \begin{proof}
		 Recall $S'\in X_g$ means shorest geodesics on $S'$ consist of a filling curve set. Then any essential curve $\alpha'$ intersects at least one shortest geodesic geodesic. We pick a shortest geodesic that intersects $\alpha'$ and denote it by $\beta'$. We let $p'$ be a point in $\alpha'\cap\beta'$. Then $\inj_{S'}(p')\le \frac{1}{2}l_{\beta'}(S')=\frac{1}{2}\sys(S')$. 
 \end{proof}
 Now we are ready to prove
 \begin{theorem}
		 For the probability of a point in $\mathcal M_g$, whose Weil-Petersson distance to $X_g$ is smaller than $0.6521 ( \sqrt{\log g} - \sqrt{7\log\log g})$, we have
		\[
				P_{WP}\left(S|d_{wp}(S,X_g)<0.6521 \left( \sqrt{\log g} - \sqrt{7\log\log g}\right)\right) \to 0
		\]
		as $g\to \infty$. 
		 \label{thm_random_wp}
 \end{theorem}
 \begin{proof}
		 		By Theorem \ref{thm_mp}, we let $c_g= \log\log g$, then 
				\begin{equation}
						P_{WP}( S\in \mathcal M_g| \sys(S)>\log\log g ) < B\frac{\log \log g}{(\log g)}. 
						\label{for_mp}
				\end{equation}

		 Let $S\in \mathcal M_g$ satisfying condition (a) and (b) in Theorem \ref{thm_nwx} and $\sys(S)\le \log \log g$. 
		 For any $S'\in X_g$,
		 by Corollary \ref{cor_sys_dist_wp}, 
		 \begin{eqnarray}
				 \label{for_dist_wp_a}
				 0.5492\, d_{wp}(S,S') &\ge& \left | \sqrt{\sys(S')}-\sqrt{\sys(S)}\right | \ge \sqrt{\sys(S')}-\sqrt{\sys(S)}\\
				 &\ge& \sqrt{\sys(S')}-\sqrt{\log\log g}. \nonumber
		 \end{eqnarray}
		 
		 On the other hand, since $S$ satisfies the condition (a) and (b) then by Lemma \ref{lem_inj_nwx} there is a curve $\alpha\subset S$, such that for any $p\in \alpha$, 
		 \begin{equation}
				 \inj_S(p)\ge \frac{1}{4}\log g-(\frac{3}{4}+\frac{\varepsilon}{2})\log\log g.  
				 \label{for_s_inj}
		 \end{equation}

		 We choose an arbitrary horizontal curve $c(t):[0,1]\to \mathcal M_{-1}$, satisfying $\pi(c(0)) = S$, $\pi(c(1)) = S'$ and $\pi(c([0,1))$ is a Weil-Petersson geodesic connecting $S$ and $S'$. Then by deforming the metric of $S$ along $c(t)$ to the metric of $S'$, we have $\alpha$ is also a well-defined essential simple closed curve on $S'$. By Lemma \ref{lem_inj_fill}, there is a point $p\in \alpha\subset S'$, such that 
				 \begin{equation}
						 \inj_{S'}(p)\le \frac{1}{2}\sys(S').
						 \label{for_sp_inj}
				 \end{equation}
		 
				 Therefore by Theorem \ref{thm_inj}, (\ref{for_s_inj}) and (\ref{for_sp_inj}), 
				 \begin{eqnarray}
						 \label{for_dist_wp_b}
						 0.3884 \,d_{wp}(S,S')&\ge& \left| \sqrt{\inj_{S}(p)}-\sqrt{ \inj_{S'}(p)} \right |\\
						\nonumber &\ge& \sqrt{\inj_{S}(p)}-\sqrt{ \inj_{S'}(p)} \\
						 \nonumber&\ge& 
						 \sqrt{\frac{1}{4}\log g-(\frac{3}{4}+\frac{\varepsilon}{2})\log\log g} - \sqrt{\frac{\sys(S')}{2}}.
				 \end{eqnarray}
				 Combining (\ref{for_dist_wp_a}) and (\ref{for_dist_wp_b}), then eliminating $\sys(S')$, we have 
				 \[
						 d_{wp}(S,S') \ge 0.6521 \left( \sqrt{\log g} - \sqrt{7\log\log g} \right). 
				 \]
				 Hence for any $S$ satisfying (a) and (b) and $\sys(S)\le \log\log g$
				 \[
						 d_{wp}(S,X_g) \ge 0.6521 \left( \sqrt{\log g} - \sqrt{7\log\log g} \right). 
				 \]
				 On the other hand, by Theorem \ref{thm_nwx} and (\ref{for_mp}),
				 \[
						 P_{WP}(S|S \text{ satisfies (a) and (b), }\sys(S)\le \log\log g)\to 1
				 \]
				 as $g\to \infty$. Therefore 
		\[
				P_{WP}\left(S|d_{wp}(S,X_g)\ge0.6521 \left( \sqrt{\log g} - \sqrt{7\log\log g}\right)\right) \to 1
		\]
		as $g\to \infty$ and the theorem holds. . 
		 
 \end{proof}
 \section{A criterion for the critical points}
 \label{sec_crit}
 The aim of this section is to prove Proposition \ref{prop_euc}, a criterion of critical points of the systole function. In \ref{sec_sub_g_inv}, we prove some properties of the tangent vector of $\mathcal{T}_g$ that is invariant under the action of its base point's automorphic group as a preparation. Then in \ref{sec_sub_criterion}, we prove Proposition \ref{prop_euc}. 
 \subsection{$G$-invariant tangent vector}
 \label{sec_sub_g_inv}
\begin{lemma}
		For Fuchsian groups $\Gamma$, $\Gamma'$, $\Gamma\vartriangleleft \Gamma'$, if $\mu\in HB(\mathbb{H}^2,\Gamma) \cap B(\mathbb{H}^2,\Gamma')$, then $\mu \in HB(\mathbb{H}^2,\Gamma')$. 
		\label{lem_harmonic}
\end{lemma}
\begin{proof}
		By (\ref{for_harmonic}), on $B(\mathbb{H}^2,\Gamma)\cap B(\mathbb{H}^2,\Gamma')$, $H:B(\mathbb{H}^2,\Gamma)\to HB(\mathbb{H}^2,\Gamma)$ and $H':B(\mathbb{H}^2,\Gamma')\to HB(\mathbb{H}^2,\Gamma)$ coincide. That is, $\forall \mu \in B(\mathbb{H}^2,\Gamma)\cap B(\mathbb{H}^2,\Gamma')$, $H[\mu] = H'[\mu]$. 

		If $\Gamma\vartriangleleft \Gamma'$, then $B(\mathbb{H}^2,\Gamma') \subset B(\mathbb{H}^2,\Gamma)$. For $\mu \in HB(\mathbb{H}^2,\Gamma)\cap B(\mathbb{H}^2,\Gamma')$, since $\mu\in HB(\mathbb{H}^2,\Gamma)$, then 
		\begin{equation}
				H[\mu] = \mu
				\label{for_project}
		\end{equation}
		because $H$ is a projection. Since $H$ is also the projection operator on $B(\mathbb{H}^2,\Gamma')$, (\ref{for_project}) holds on $B(\mathbb{H}^2,\Gamma')$ and then $\mu \in HB(\mathbb{H}^2,\Gamma')$. 
\end{proof}
 We assume $R$ is a closed genus $g$ Riemann surface and $\Gamma$ is a Fuchsian group model of $R$, $T(R)$ and $T(\Gamma)$ are the Teichm\"uller space with base point $R$ and $\Gamma$ respectively. 
 We let $g\in Aut(R)$. Then $g$ acts on $T(R)$ the action is given by
 \[ [S,f]\mapsto [S,f\circ g^{-1}] \]
 and also $g$ acts on $T(\Gamma)$, this action is given by 
 \begin{equation}
		 [w]\mapsto [\alpha\circ w\circ \tilde{g}^{-1}]. 
		 \label{for_g_action_on_gamma}
 \end{equation}
 Here $\tilde{g}$ is a lift of $g$ onto $\mathbb{H}^2$ and $\alpha\in PSL_2(\mathbb{R})$ satisfies $\alpha\circ w\circ \tilde{g}^{-1}$ fixing $0$, $1$ and $\infty$. 

 On the tangent space $HB(\mathbb{H}^2,\Gamma)$, the induced tangent map is given by 
 \[
		 g_*:\mu \mapsto (\mu\circ \tilde{g}^{-1})\frac{\overline{(\tilde{g}^{-1})'}}{(\tilde{g}^{-1})'}. 
 \]

 If $\mu \in HB(\mathbb{H}^2,\Gamma)$ is $g$-invariant, namely $g_*\mu = \mu$ then 
 \[
		 \mu =  (\mu\circ \tilde{g}^{-1})\frac{\overline{(\tilde{g}^{-1})'}}{(\tilde{g}^{-1})'}. 
 \]
 By (\ref{for_fuchsian}) 
 there is a Fuchsian group $\Gamma'$, containing $\Gamma$ and $\tilde{g}$, such that $\mu\in B(\mathbb{H}^2,\Gamma')$. 
 Then by Lemma \ref{lem_harmonic}, $\mu\in HB(\mathbb{H}^2,\Gamma')$. Immediately we have
 \begin{lemma}
		 $\mu \in HB(\mathbb{H}^2,\Gamma)$ is $g$-invariant if and only if $\mu \in HB(\mathbb{H}^2,\Gamma')$. 
		 Here $\Gamma'$ is a Fuchsian model of $R/\langle g\rangle$ containing $\Gamma$. ($\langle g\rangle$ is the subgroup of $Aut(R)$ generated by $g$. )

		 \label{lem_g_invariant}
 \end{lemma}

More generally, for a subgroup $G$ of $Aut(G)$, we have 
 \begin{lemma}
		 $\mu \in HB(\mathbb{H}^2,\Gamma)$ is $G$-invariant if and only if $\mu \in HB(\mathbb{H}^2,\Gamma')$. 
		 Here $\Gamma'$ is a Fuchsian model of $R/G$ containing $\Gamma$.

		 \label{lem_group_invariant}

 \end{lemma}

 \subsection{The criterion}
 \label{sec_sub_criterion}

 For genus $g$ surface $S_g$, we assume $G$ is a finite subgroup of $MCG(S_g)$, and $\rho$ is a marked hyperbolic structure on $S_g$ such that $\Sigma_g = (S_g,\rho) \in \mathcal{T}_g$ . 
 Then we define $X_g^G \subset \mathcal{T}_g$, the hyperbolic surface admits $G$ action as 
 \[
		 X_g^G = \left\{ \Sigma_g=(S_g,\rho) \in \mathcal{T}_g| G\le Aut(\Sigma_g) \right\}.
 \]

 The following proposition characterize the local coordinate of $X_g^G$. 
 \begin{proposition}
		 For $R\in X_g^G$, let $\Gamma$ be a Fuchsian model of $R$ and $\Gamma'$ be a Fuchsian model of $R/G$ such that $\Gamma\vartriangleleft\Gamma'$. We assume that 
		 \[
				 U = \left\{ [w^\mu]\in T(\Gamma)| \mu\in HB(\mathbb{H}^2,\Gamma'), \| \mu\|_\infty <\varepsilon_0\right\}
		 \]
		 is a subset of $T(\Gamma)$, consisting of surfaces induced by $G$-invariant tangent vectors of $T(\Gamma)$ at $\Gamma$. Then 
		 \begin{enumerate}
				 \item $U\subset X_g^G$. 
				 \item $U$ is homeomorphic to a domain in $\mathbb{R}^k$, where $k = \dim HB(\mathbb{H}^2,\Gamma')$. 
				 \item $HB(\mathbb{H}^2,\Gamma') = T_{[id]} U$. 
		 \end{enumerate}
		 \label{prop_local_xgg}
 \end{proposition}
 \begin{proof}
		 (1) For $[w^\mu]\in U$ and $g\in G$, the action of $g$ on $[w^\mu]$ is given by 
		 \[
				 g:[w^\mu] \mapsto [\alpha\circ w^\mu\circ \tilde{g}^{-1}]
		 \]
		 as (\ref{for_g_action_on_gamma}). The Beltrami differential of $g([w^\mu])$ is 
 \[
		  (\mu\circ \tilde{g}^{-1})\frac{\overline{(\tilde{g}^{-1})'}}{(\tilde{g}^{-1})'}. 
 \]
 Since $\mu$ is $G$-invariant 
 \[
		  \mu = (\mu\circ \tilde{g}^{-1})\frac{\overline{(\tilde{g}^{-1})'}}{(\tilde{g}^{-1})'}. 
 \]
 Then by the uniqueness of $w^\mu$ in Theorem \ref{thm_qc_exist}, \[
		 [w^\mu ] = [\alpha\circ w^\mu \circ \tilde{g}^{-1}]. 
 \]
 Therefore $[w^\mu]$ is $G$-invariant and $[w^\mu] \in X_g^G$.

 (2) Let $V= \left\{\mu\in HB(\mathbb{H}^2,\Gamma')| \| \mu\|_\infty  < \varepsilon_0\right\}$. Consider the map from $V$ to $U$:
 \begin{eqnarray*}
		 F: V &\to &U\\
		 \mu & \mapsto & w^\mu.
 \end{eqnarray*}
 Then $F$ is continuous by Proposition \ref{prop_qc_continuous} and is injective by the uniqueness of $w^\mu$ in Theorem \ref{thm_qc_exist}. By definition, $F(V)=U$. Then by the domain invariance theorem, $F$ is a homeomorphism between $U$ and $V$ and $V$ is a domain in $\mathbb{R}^k$. 

 (3) By (2), $U$ is an open set in $T(\Gamma')$. On the other hand, $T(\Gamma')$ is embeded in $T(\Gamma)$. This embedding is induced by the group inclusion $\Gamma\vartriangleleft\Gamma'$. Therefore, $HB(\mathbb{H}^2,\Gamma') = T_{[id]}T(\Gamma') = T_{[id]}U$. 
 \end{proof}

After these preparations, we obtain the following proposition. It is a generalization to \cite[Theorem 37]{schaller1999systoles} and \cite[Proposition 6.3]{fortier2019hyperbolic}.

\begin{proposition}
		If $R\in X_g^G $ realizes the maxima of the systole function on $X_g^G$, i.e. 
		\[
				\sys R \ge \sys S, \forall S \in X_g^G.
		\]
		then $R$ is a critical point of the systole function in $\mathcal{T}_g$. 
		\label{prop_euc}
\end{proposition}

\begin{proof}
		We assume that $R$ realizes the maxima of $\sys$ on $X_g^G$, $S(R)$ is the set of systoles of $R$, $\Gamma$ is a Fuchsian model of $R$ and $\Gamma'\vartriangleright\Gamma$ is a Fuchsian model of $R/G$. 

		For $\mu \in HB(\mathbb{H}^2,\Gamma)$, if $\forall \alpha \in S(R)$, $\mathrm d l_\alpha(\mu)\ge 0$, we consider $\nu = \sum_{g\in G}g_* \mu$ then we have
		\begin{equation}
				\mathrm d l_\alpha(\nu) = \mathrm d l_\alpha\left( \sum_{g\in G}g_* \mu \right) = \sum_{g\in G}\mathrm d l_\alpha(g_*\mu) = \sum_{g\in G} \mathrm d l_{g(\alpha)}(\mu) \ge \mathrm d l_{\alpha}(\mu) \ge 0. 
				\label{for_sum}
		\end{equation}
		\begin{remark}
				This formula is the (6.1) in \cite{fortier2019hyperbolic}. 
		\end{remark}

		The vector $\nu = \sum_{g\in G}g_* \mu$ is in $HB(\mathbb{H}^2,\Gamma')$. We consider $U = \left\{ w^\nu|\nu\in HB(\mathbb{H}^2,\Gamma'), \|\nu\|_\infty<\varepsilon_0 \right\}$. By Proposition \ref{prop_local_xgg} (2), $U$ is a domain and therefore a manifold. If $U$ is small enough, for any point $S\in U$, there exists $\alpha\in S(R)$ such that $\sys(S) = l_\alpha(S)$. The Hessian of $l_\alpha|_U$ is positively definite since the Hessian of $l_\alpha$ is positively definite. Then by Theorem \ref{thm_akrout}, $\sys|_U$ is a topological Morse function. 

		Since $R$ realizes the maxima of $\sys|_{X_g^G}$, $R$ realizes the maxima of $\sys|_U$ and $R$ is a critical point of $\sys|_U$. By Proposition \ref{prop_local_xgg}(3), $HB(\mathbb{H}^2,\Gamma')$ is the tangent space of $U$ at the base point. By Definition \ref{defi_eut}, for $\nu\in HB(\mathbb{H}^2,\Gamma')$, if $\mathrm d l_{\alpha}(\nu)\ge 0$, $\forall \alpha\in S(R)$, then $\mathrm d l_{\alpha}(\nu) = 0$, $\forall \alpha\in S(R)$. 

		Therefore, by (\ref{for_sum}), for $\mu\in HB(\mathbb{H}^2,\Gamma)$, if $\mathrm d l_{\alpha}(\mu)\ge 0$, $\forall \alpha\in S(R)$, then $\mathrm d l_{\alpha}(\mu) = 0$, $\forall \alpha\in S(R)$. By Definition \ref{defi_eut}, $R$ is an eutactic surface and therefore a critical point of the systole function. 
\end{proof}

		\section{Small distance}
		\label{sec_small}
		\subsection{Construction of $S_g^1$ and $S_g^2$}
		\label{sec_construct}
		The surface $S^1_g$ is initially constructed in \cite{anderson2011small}, while $S^2_g$ is initially constructed in \cite{bai2019maximal}. We briefly construct these two surfaces for completeness and this construction implies how to atbain the Teichm\"uller distance between the two surfaces. 

We first construct a family of genus $g$ hyperbolic surfaces denoted by $\{S_g(c,t)\}$, each surface in this family is determined by two parameters $c$ and $t$
		for $c>0$ and $0\le t\le c/2$. The example $S^1_g$ is a $S_g(c_1,0)$ surface for some $c_1>0$ while the example $S^2_g$ is a $S_g(c_2,t_2)$ surface for some $c_2,t_2>0$

		We let $n\ge3$ and pick two 
		isometric right-angled hyperbolic polygons with $2n$ edges, admitting an order $n$ rotation. Two such polygons are able to be glued to an $n$-holed sphere admitting the order $n$ rotation extended from the polygons. 
		By this rotation, all boundary curves of this $n$-holed sphere have the same length. 
		The geometry of the $n$-holed sphere is determined by the length of its boundary curves. We denote this length as $c$ and the corresponding $n$-holed sphere by $S(c)$.  We call the boundary curves of $S(c)$ {\emph cuffs} and the edges of the polygons contained in the interior of $S(c)$ {\emph seams}. By the rotation symmetry, all seams also have the same length. 

		We pick two isometric $n$-holed spheres and glue them along their boundary curves, getting a closed surface. As shown in Figure \ref{fig_exp_1_2}, when gluing the two $n$-holed sphere,  we require that every cuff of one of the $n$-holed spheres is identified with a cuff in the other $n$-holed sphere, and 
		every seam of one $n$-holed sphere is half of a closed curve (denoted $\alpha_k$ for $k=1,2,\dots,n$) while the other half of $\alpha_k$ is a seam in the other $n$-holed sphere. 

		\begin{figure}[htbp]
				\centering
				\includegraphics{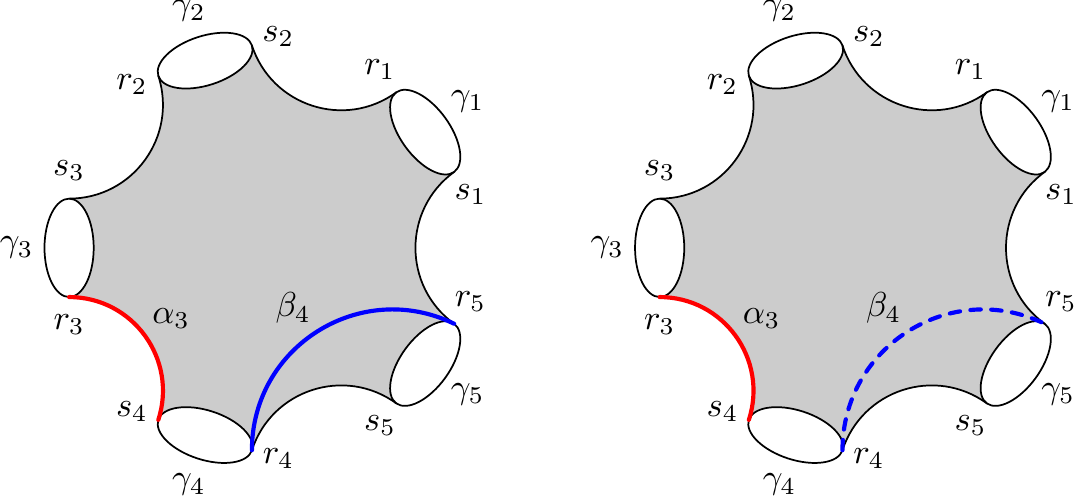}
				\caption{}
				\label{fig_exp_1_2}
		\end{figure}
		This constructed surface has genus $g=n-1$ and 
		admits the order $g+1$ rotation extended from the rotation on the $n$-holed spheres. 
		Geometry of this closed surface is determined by the cuff length $c$ of the two $n$-holed spheres. We denote this surface by $S_g(c,0)$. 
		
		Now we begin to construct $S_g(c,t)$ for $t>0$. In $S_g(c,0)$, we denote a cuff of the $n$ holed sphere by $\gamma_k$ for $k=1,2,\dots,g+1$. We conduct a Fenchel-Nielsen deformation on $S(c,0)$ along all the $\gamma_k$ curves. The time and orientation of the deformation at every $\gamma_k$ is the same so that the resulting surface preserves the order $g+1$ otation on $S_g(c,0)$. We assume the time of the Fenchel-Nielsen deformation is $t$ and denote the resulting surface as $S_g(c,t)$. 

		There is a $c_1>0$ such that on the surface $S_g(c_1,0)$, $l(\alpha_k)= l(\gamma_k)$. This surface is the surface $S^1_g$. The shortest geodesics of $S^1_g$ consist of $\alpha_k$ and $\gamma_k$ for $k=1,2,\dots,g+1$ by the proof of \cite[Theorem 3]{anderson2011small}. 

		In a surface $S_g(c,t)$, we let $\beta_k$ be the image of $\alpha_k$ by a Dehn twist along $\gamma_{k+1}$ (Figure \ref{fig_exp_1_2}). Orientation of this Dehn twist is required to be opposite to the Fenchel-Nielsen deformation so that length of $\beta_k$ decreases as $t$ increases when $t$ is closed to $0$. 

		There is a pair $(c_2,t_2)$ such that on the surface $S_g(c_2, t_2)$, $l(\alpha_k)=l(\beta_k) = l(\gamma_k)$. This surface is the surface $S^2_g$. The shortest geodesics of $S^2_g$ consist of $\alpha_k$, $\beta_k$ and $\gamma_k$ for $k=1,2,\dots,g+1$ by \cite[Proposition 4]{bai2019maximal}. 

		\subsection{Symmetry on $S(c,t)$ }

		By the construction of the surface $S_g(c,t)$, for any $c>0$ and $0\le t\le c/2$, $S(c,t)$ admits an order $g+1$ rotation rotation, denoted $\sigma$ (Figure \ref{fig_exp_1_3}). Besides this rotation, $S_g(c,t)$ admits an order $2$ rotation $\tau$ exchanging the two $n$-holed spheres (Figure \ref{fig_exp_1_6}) and an order $2$ rotation $\varsigma$ extended from an order $2$ rotation on one of the $n$-holed spheres. We denote the symmetric group on $S(c,t)$ generated by $\sigma$, $\tau$ and $\varsigma$ by $G$. 

		\begin{figure}[htbp]
				\begin{subfigure}{.45\linewidth}
						\centering
						\includegraphics{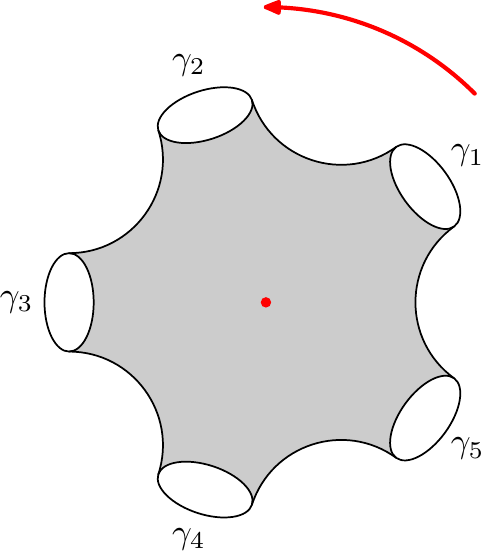}
						\caption{The rotation $\sigma$.}
						\label{fig_exp_1_3}
				\end{subfigure}
				\begin{subfigure}{.45\linewidth}
						\centering
						\includegraphics{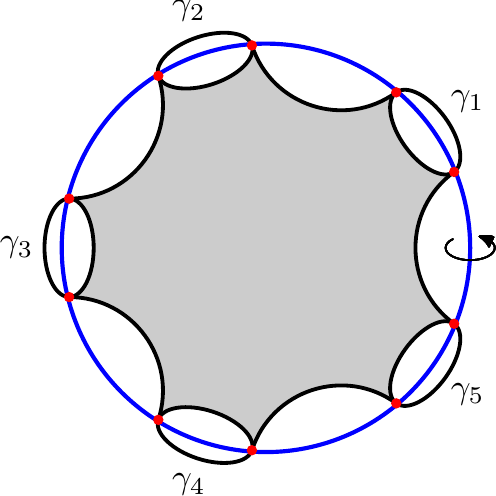}
						\caption{The rotation $\tau$.}
						\label{fig_exp_1_6}
				\end{subfigure}

				\begin{subfigure}{.45\linewidth}
						\centering
						\includegraphics{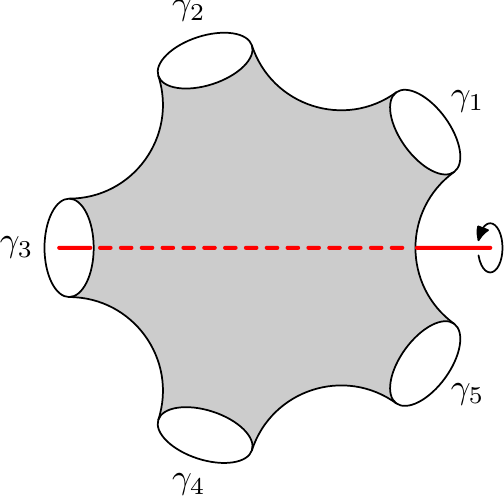}
						\caption{The rotation $\varsigma$.}
						\label{fig_exp_1_4}
				\end{subfigure}
				\begin{subfigure}{.45\linewidth}
						\centering
						\includegraphics{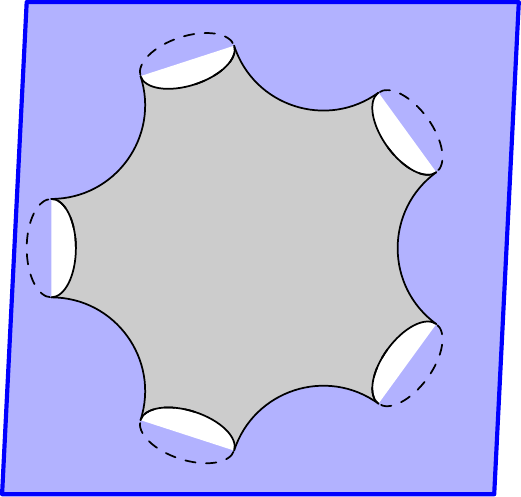}
						\caption{The reflection $\rho$.}
						\label{fig_exp_1_5}
				\end{subfigure}
				\caption{}
				\label{fig_iso}
		\end{figure}

		On the surface $S_g(c,0)$, there is a reflection $\rho$ extended from the reflection on one of the $n$-holed spheres exchanging the two polygons of the $n$-holed sphere (Figure \ref{fig_exp_1_5}). The symmetric group generated by $\sigma$, $\tau$, $\varsigma$ and $\rho$ is denoted by $\bar{G}$. 

		\begin{remark}
				For a reflection on the $n$-holed sphere, this reflection can be extended to the whole surface $S_g(c,t)$ only if $t=0$ or $t=c/2$. 

				The reflection on $S_g(c,c/2)$ (denoted by $\rho_{c/2}$) is not conjugate to $\rho$. This is because their fixed point sets are different. The fixed points of $\rho$ on $S_g(c,0)$ consist of $g+1$ curves (the $\beta_k$ curves), whille fixed points of $\rho_{c/2}$ consist of one curve (when $g$ is even) or two curves (when $g$ is odd), see Figure \ref{fig_conj}. 
				\label{rem_conj}
		\end{remark}
		\begin{figure}[htbp]
				\begin{subfigure}{.45\linewidth}
						\centering
						\includegraphics{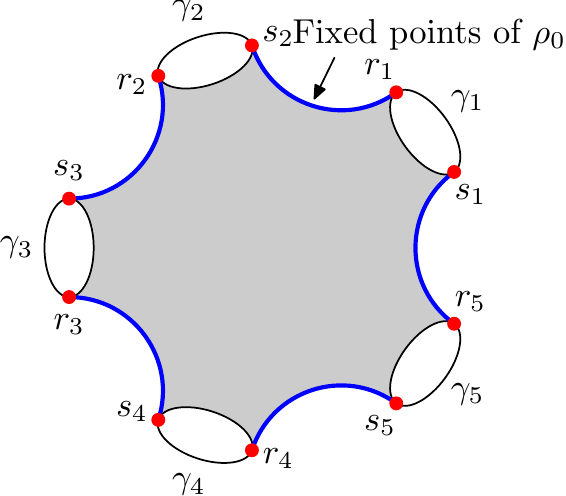}
						\caption{The fixed point set of $\rho$ on $\Sigma_g(c,0)$.}
						\label{fig_exp_1_8}
				\end{subfigure}
				\begin{subfigure}{.45\linewidth}
						\centering
						\includegraphics{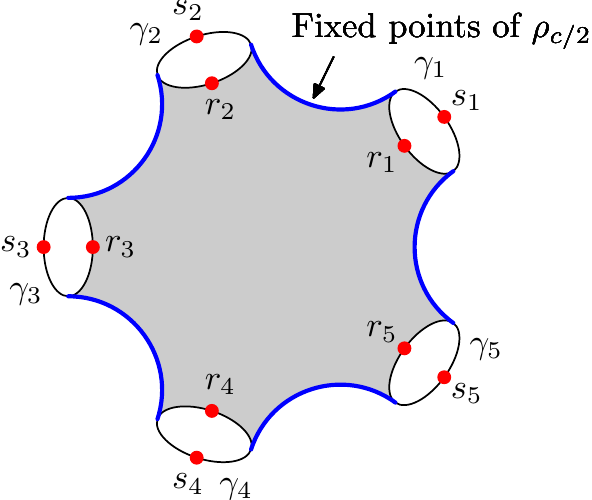}
						\caption{The fixed point set of $\rho_{c/2}$ on $\Sigma_g(c,c/2)$.}
						\label{fig_exp_1_9}
				\end{subfigure}

				\caption{}
				\label{fig_conj}
		\end{figure}

		The surface $S_g^1$ has been proved to be a critical point of the systole function, see \cite[Example 4.2, Proposition 6.3]{fortier2019hyperbolic}. 

		On the other hand, it is proved in \cite{bai2019maximal} that the surface $S_g^2$ is the surface with the maximal systole among the surfaces admitting the action of $G$. 
		Then immediately by Proposition \ref{prop_euc}, $S_g^2$ is a critical point of the systole function. 

		Hence we have 
		\begin{proposition}
				The surfaces $S_g^1$ and $S_g^2$ are critical points of the systole function. 
				\label{prop_sigma12}
		\end{proposition}

		\subsection{Distance}
		The aim of this subsection is to bound the Teichm\"uller distance between $S^1_g$ and $S^2_g$. 

		To get an upperbound of $d_{\mathcal T}(S^1_g,S^2_g)$, we need an intermediate surface. We assume the parameter $S^1_g = S_g(c_1,0)$ and $S^2_g = S_g(c_2,t_2)$. Then the intermediate surface is the surface $S_g(c_2,0)$. Distance between $S^1_g$ and $S^2_g$ is bounded from above by the sum of $d_{\mathcal T}(S_g^1,S(c_2,0))$ and $d_{\mathcal T}(S(c_2,0), S_g^2)$. 
		\subsubsection{The distance between $S_g^1=S(c_1,0)$ and $S(c_2,0)$}
		We estimate this distance in three steps:

		(1) Prove $\left\{ S(c,0)|c>0 \right\}$ is a Teichm\"uller geodesic induced by a Jenkins-Strebel differential on some surface $S(c,0)$. 

		(2) Calculate distance between two points on the Teichm\"uller geodesic. This distance is given by a ratio of extremal lengths of a curve on the two surfaces.

		(3) Estimate this distance expressed in extremal lengths by hyperbolic lengths using Maskit's Theorem (Theorem \ref{thm_maskit}). 

		For $c>0$, on the surface $S(c,0)$, we consider the cuffs of the $n$-holed spheres in $S(c,0)$, namely $\left\{ \gamma_k \right\}_{k=1}^{g+1}$ and assign to each $\gamma_k$ a positive number $b$. Then by Theorem \ref{thm_bi}, $\left\{ (\gamma_k,b) \right\}_{k=1}^{g+1}$ induces a quadratic differential $q$ on $S(c,0)$. 

		\begin{lemma}
				The quadratic differential $q\in QD(S_g(c,0))$ is invariant under the action of $\bar{G}$.
				\label{lem_inv}
		\end{lemma}
		\begin{proof}
				For $g\in \bar{G}$, the quadratic form $g^*q$ is induced by the set $\left\{ (g^{-1}(\gamma_k),b) \right\}_{k=1}^{g+1}$. By the action of $\bar{G}$ on $S_g(c,0)$, $\left\{ (g^{-1}(\gamma_k),b) \right\}_{k=1}^{g+1} =\left\{(\gamma_k,b) \right\}_{k=1}^{g+1} $. 
				Therefore $g^*q = q$ and $q$ is invariant. 
		\end{proof}

		We consider the Teichm\"uller geodesic induced by $(S(c,0),q)$. 
		\begin{lemma}
				We denote the  Teichm\"uller geodesic induced by $(S(c,0),q)$ as $l$. Then the Teichm\"uller geodesic $l$
				coincides with the curve $\left\{ S_g(c,0)|c>0 \right\}$. 
				\label{lem_symm}
		\end{lemma}
		\begin{proof}
				Since $q$ is $\bar{G}$-invariant by Lemma \ref{lem_inv}, For any surface $S'\in l$  
				the Beltrami coefficient of the Teichm\"uller map $f:S(c,0)\to S'$ is $t\frac{\bar{q}}{q}$ for some $t\in (0,1)$. Hence this Beltrami coefficient is $\bar{G}$-invariant.  Then by Proposition \ref{prop_local_xgg} (1), 
				$\bar{G}$ isometrically acts on $S'$ by 
				\[f\circ g \circ f^{-1}: S'\to S' \]
				for any $g\in \bar{G}$. 

				Consider the set of cuffs of the $n$-holed spheres on $S_g(c,0)$, denoted $\left\{ \gamma_k \right\}_{k=1}^{g+1}$. Its image $\left\{ f(\gamma_k) \right\}_{k=1}^{g+1}$ in $S'$ cuts $S'$ into two $n$-holed spheres. Then $\bar{G}$ isometrically acts on these two $n$-holed spheres as $\bar{G}$ acts on the two $n$-holed spheres in $S(c,0)$ divided by $\left\{ \gamma_k \right\}_{k=1}^{g+1}$. Then $S'$ is a $S(c',0)$ surface, where $c'$ is the length of $f(\gamma_k)$ on $S'$. 

				Therefore the Teichm\"uller geodesic $l$ is contained in the curve $\left\{ S_g(c,0)|c>0 \right\}$. Then by the completeness of Teichm\"uller geodesics, $\left\{ S_g(c,0)|c>0 \right\}$ coincides with $l$. 
		\end{proof}

				Distance between two points on a Teichm\"uller geodesic has been given by (\ref{for_dist_extremal}). Now we are ready to acomplish the third step. 

		Finally we have 
		\begin{proposition}
				For $S_g^1=S_g(c_1,0)$ and $S_g{(c_2,0)}$, 
				 we have 
				\[
						d_{\mathcal{T}}(\Sigma_g^1,\Sigma_g^{1,2})\le 0.65. 
				\]
				\label{prop_dist_1}
		\end{proposition}
		\begin{proof}
				Recall that $c_1$ and $c_2$ are the systole of $S_g^1$ and $S_g^2$ respectively. Then $c_1 = 4\arcsinh\sqrt{\cos\frac{\pi}{g+1}}$ by \cite{anderson2011small} and $c_2$is given by the formula in \cite[Theorem 1]{bai2019maximal}. Then we use the following lemma to get the Teichm\"uller distance. 

		\begin{lemma}
				For hyperbolic surfaces $S_g(c_1,0)$ and $S_g{(c_2,0)}$ with $c_1<c_2$, the Teichm\"uller distance between these two surfaces is bounded from above by 
				\[
						\frac{1}{2}\log \frac{\pi c_2}{2\theta c_1}, 
				\]
				where 
				\[
						\cos \theta = \left( 1+\frac{\cos^2\frac{\pi}{g+1}}{\sinh^2 \frac{c_2}{4}} \right)^{-\frac{1}{2}}
				\]. 
				\label{lem_hyp_dist}
		\end{lemma}
		\begin{proof}
				For $i=1,2$, we let $\{ \gamma_k^{(i)} \}_{k=1}^{g+1}$ be the cuffs in $S_g(c_i,0)$, $q_i\in QD(S_g(c_i,0))$ be the quadratic differential induced by $\{ (\gamma_k^{(i)},b) \}_{k=1}^{g+1}$ for some $b>0$ and $R_k^{(i)}$ be the characteristic ring domain of $\gamma_k^{(i)}$. Then by Theorem \ref{thm_maskit}
				\begin{equation}
						\Ext_{\gamma_k^{(1)}}(R_k^{(1)})\ge \Ext_{\gamma_k^{(1)}}(S_g(c_1,0))\ge \frac{l(\gamma_k^{(1)})}{\pi} = \frac{c_1}{\pi}.
						\label{for_c1}
				\end{equation}

				The set of characteristic ring domains $\{ R_k^{(2)} \}_{k=1}^{g+1}$ is invariant under the $\bar{G}$-action. Then by the symmetry of $\bar{G}$, in $S_g(c_2,0)$, the ring domains $R_k^{(2)}$, $k=1,\dots,g+1$ are bounded by the hyperbolic geodesics, connecting a center of the $2n$-polygons and a middle point of the seams (Figure \ref{fig_exp_1_11}), otherwise $\{ R_k^{(2)} \}_{k=1}^{g+1}$ is not $\bar{G}$-invariant. 
				\begin{figure}[htbp]
						\centering
\begin{minipage}[t]{0.49\textwidth}
\centering
						\includegraphics{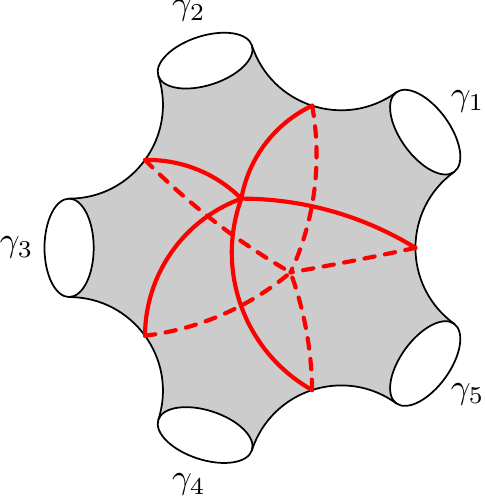}
						\caption{Charateristic ring domains}
						\label{fig_exp_1_11}
\end{minipage}
\begin{minipage}[t]{0.49\textwidth}
\centering
						\includegraphics{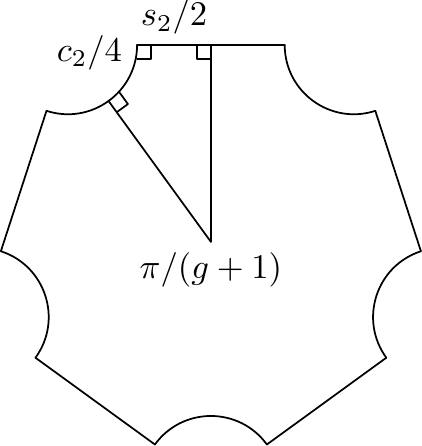}
						\caption{Calculate $s_2/2$}
						\label{fig_octagon_6}
\end{minipage}
				\end{figure}

				Therefore, if the seam length of $n$-holed spheres of $S_g(c_2,0)$ is $s_2$, then the collar $C_k$ of $\gamma_k^{(2)}$ with width $s_2/s$ is contained in the characteristic ring domain $R_k^{(2)}$. 

				The seam length $s_2$ is given by the formula in trirectangle (\ref{for_trirect_1}), see Figure \ref{fig_octagon_6}
				\begin{equation}
						\sinh \frac{s_2}{2}\sinh \frac{c_2}{4} = \cos \frac{\pi}{g+1}. 
						\label{for_s2}
				\end{equation}
				Therefore, by Theorem \ref{thm_maskit}
				\begin{equation}
						\Ext_{\gamma_k^{(2)}}(R_k^{(2)})\le\Ext_{\gamma_k^{(2)}}(C_k)\le \frac{l(\gamma_k^{(2)})}{2\arccos\frac{1}{\cosh({s_2}/{2})}} = \frac{c_2}{2\arccos\frac{1}{\cosh({s_2}/{2})}}
						\label{for_c2}
				\end{equation}
				By combining (\ref{for_dist_extremal}) (\ref{for_c1}) (\ref{for_c2}) and (\ref{for_s2}), this Lemma holds. 

		\end{proof}
		Proposition \ref{prop_dist_1} follows immediately by Lemma \ref{lem_hyp_dist}. 
		\end{proof}

		\subsubsection{The distance between $S_g(c_2,0)$ and $ S_g^2=S_g(c_2,t_2)$}
		Recall that $S_g(c_2,t_2)$ is obtained from $S_g(c_2,0)$ by a Fenchel-Nielsen deformation along the cuffs $\left\{ \gamma_k \right\}_{k=1}^{g+1}$ in $S(c_2,0)$ with time $t_2$. 
		If we can construct a homeomorphism $h:S_g(c_2,0)\to S_g(c_2,t_2)$ such that for a collar $C_k$ of $\gamma_k$ where $k=1,2,\dots,g+1$, $h$ is an isometry outside all these collars, then Teichm\"uller distance between $S_g(c_2,0)$ and $S_g(c_2,t_2)$ is bounded from above by the  $\frac{1}{2}\log K(h)$ fo the dilatation $K(h)$ and $K(h)$ equals $K(h|_{C_k})$ the dilatation of $h$ restricted on some $C_k$. 

		\begin{proposition}
				For $\Sigma_g^2$ and $\Sigma_g^{1,2}$, we have 
				\[
						d_{\mathcal{T}}(\Sigma_g^2,\Sigma_g^{1,2})\le 1.6450. 
				\]
				\label{prop_dist_2}
		\end{proposition}
		\begin{proof}
				We prove this proposition by constructing the homeomorphism $h$ and calculate its dilatation on the largest colar of $\gamma_k$. 

				We let $C_k$ be the collar of $\gamma_k$ with the width $s_2/$, where $s_2$ is the seam length of the $n$-holed spheres as in Lemma \ref{lem_hyp_dist}. The homeomorphism $h$ on $C_k$ is described in the Figure \ref{fig_homeo_1}. A geodesic $l$ orthogonal to the core curve $\gamma_k$ is always mapped to a geodesic $l'$. The line $l$ is required to intersect $l'$ at a point $p$ on $\gamma_k$. 
				The projection of one of the end points of $l'$ (denoted $p'$) is required to has distance $t_2/2$ to $p$. 

				\begin{figure}[htbp]
						\centering
\begin{minipage}[t]{0.49\textwidth}
						\includegraphics{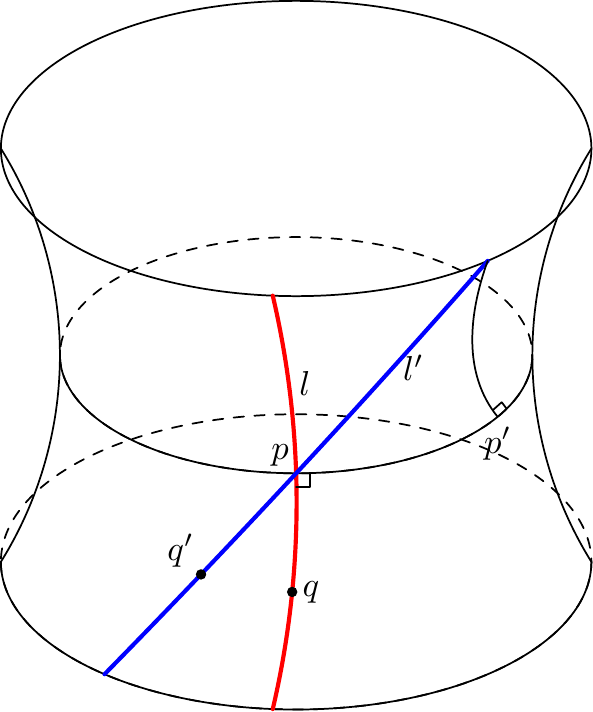}
						\caption{The homeomorphism $h|_{C_k}$.}
						\label{fig_homeo_1}
\end{minipage}
\begin{minipage}[t]{0.49\textwidth}
						\centering
						\includegraphics{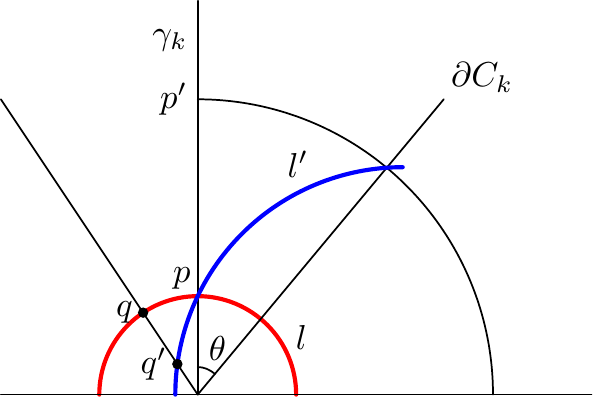}
						\caption{The lift of ${C_k}$ to $\mathbb{H}^2$.}
						\label{fig_homeo_2}
\end{minipage}
				\end{figure}

				We let $h$ outside the collars be an isometry on this surface of $S_g(c_2,0)$, then the homeomorphism $h$ maps $S_g(c_2,0)$ to $S_g(c_2,t_2)$ by the construction on the collars. 

				The rest of the proof is on the calculation of $K(h)$ on the collar $C_k$. 
				To calculate this dilatation, we lift $C_k$ on the upper half plane $\mathbb H^2$ (Figure \ref{fig_homeo_2}). 

				We lift $\gamma_k$ to the $y$-axis, assuming $p=i$ and $p'=ie^{t_2/2}$. The collar of $\gamma_k$ with width is lifted to a strip $\left\{ re^{i\varphi}\in \mathbb H^2| -\theta+\pi/2< \varphi<\theta+\pi/2 \right\}$, where 
				\begin{equation}
						\cos \theta = \frac{1}{\cosh \frac{s_2}{2}}. 
						\label{for_width}
				\end{equation}
				In this strip, $l$ is the unit circle and $l'$ is the geodesic connecting $i$ and $\exp(t_2/2+i\sin\theta)$. 

				The homeomorphism $h$ can be expressed in the form 
				\[
						h(re^{i\varphi}) = r\Phi(\varphi)e^{i\varphi}.
				\]
				When $r=1$, $h$ maps $l$ to $l'$ in Figure \ref{fig_homeo_2}. By this requirement we can calculate that 
				\begin{equation}
						\Phi(\varphi) = \sinh \left( \frac{t_2}{2} \right) \frac{\cos \varphi}{\sin \theta} + \sqrt{\sinh^2 \left( \frac{t_2}{2} \right) \frac{\cos^2 \varphi}{\sin^2 \theta}+1}. 
						\label{for_phi}
				\end{equation}
				The dilatation $K(h)$ is given by 
				\begin{equation}
						K(h) = \frac{|h_z|+|h_{\bar{z}}|}{|h_z|-|h_{\bar{z}}|} = \frac{\sqrt{\Phi^2+\frac{1}{4}\Phi'^2}+\frac{1}{2}|\Phi'|}{\sqrt{\Phi^2+\frac{1}{4}\Phi'^2}-\frac{1}{2}|\Phi'|}.
						\label{for_beltremi}
				\end{equation}
				Here $z=re^{i\varphi}$ and $\bar{z} = re^{-i\varphi}$. 

				Combining (\ref{for_beltremi}), (\ref{for_phi}) (\ref{for_width}) (\ref{for_s2}) and the formula for $(c_2,t_2)$ in \cite[Theorem 1]{bai2019maximal}, we obtain the distance $d_{\mathcal T}(S_g(c_2,0),S_g(c_2,t_2))\le\frac{1}{2}\log K(h)\le 1.6450$ (see Figure \ref{fig_dist_212}). 

				\begin{figure}[htbp]
						\centering
						\includegraphics[scale=0.5]{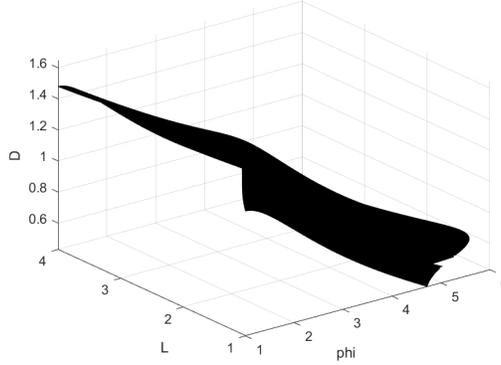}
						\caption{The dilation of $h$. Here $D$ is $K(h)$, $phi$ is the variable $\varphi$ and $L$ is $4\cos^2 \frac{\pi}{g+1}$. }
						\label{fig_dist_212}
				\end{figure}
		\end{proof}

		Hence by Proposition \ref{prop_dist_1} and \ref{prop_dist_2}, we have 
		\begin{theorem}
				For any $g\ge2$, 
				\[
						d_{\mathcal T}(S_g^1,S_g^2)\le 2.3. 
				\]
				\label{thm_dist_small}
		\end{theorem}
		\section{Large distance}
		\label{sec_large}

		\subsection{The $S_g^3$ surface}
		We take the $X(\Gamma)$ surface in \cite{fortier2020local} when $n=2$ as the surface $S_g^3$. We briefly describe this surface for completeness. 

		We consider the $4$-holed sphere admitting the order $4$-rotation. We pick infinitely many copies of the $4$-holled sphere $\left\{ P_k \right\}_{k=-\infty}^{+\infty}$ and glue them together into a surface $S_{\infty}$ with infinite genus as shown in Figure \ref{fig_line}. 
		\begin{figure}[htbp]
				\centering
				\includegraphics{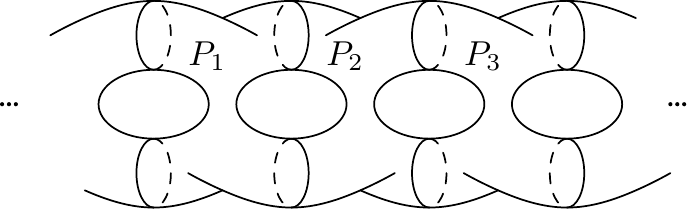}
				\caption{The surface $S_{\infty}$}
				\label{fig_line}
		\end{figure}

		The surface $S_{\infty}$ admits an isometric action $\psi:S_{\infty} \to S_{\infty} $, whic takes every $P_k$ to $P_{k+1}$. The surface $S^3_{g}$ is the quotient $S_{\infty}/\langle \psi^{g-1}\rangle$. When $g\ge 13$, $S_g^3$ is a local maximal point of the systole function.

		\subsection{The distance between $S_g^1$ and $S_g^3$}
		This istance is obtained from diameter comparison. Diameter of $S_g^3$ is comparable with $g$ while diameter of $S_g^1$ is comparable with $\log g$. Then distance between these two surfaces is comparable with $\log g$ by the method in the proof of \cite[Lemma 5.1]{rafi2013diameter}. 

		\begin{proposition}
				For the diameter of the surface $S_g^3$, we have 
				\[
						\diam (S_g^3)\ge 
						0.6\left\lfloor \frac{g-5}{2}\right\rfloor. 
				\]
				\label{prop_diam_large}
		\end{proposition}
		\begin{proof}
				By the construction, the surface $S_g^3$ consists of $g-1$ four-holed spheres $P_k$, $k=1,2,\dots,g-1$. 

				When $g\ge 5$, for any $x\in P_k$ and $y\in P_{k+2}$ for some $k$, a curve connecting $x$ and $y$ must pass at least one of the four-holed spheres except $P_k$ and $P_{k+2}$. Without loss of generality, we assume this curve passes $P_{k+1}$, then this curve, if given an orientation, enters $P_{k+1}$ at one cuff and leaves $P_{k+1}$ at another cuff. Therefore, 
				$d(x,y)$ is bounded from below by the distance between neighboring cuffs of $P_{k+1}$. We denote this distance by $d$. Then inductively, when $k\le(g-1)/2$, distance between $x\in P_1$ and $y\in P_k$ is at least $d\lfloor\frac{g-1}{2} -2 \rfloor$. Hence 
				\[
						\diam(S_g^3)\ge\left \lfloor \frac{g-1}{2} - 2\right \rfloor.
				\]

				The rest of this proof is to calculate $d$. The distance $d$ is actually the seam length of the four-holed spheres. 
				The seam length $d$ is determined by the cuff length (denoted $c$) of the four-holed sphere by (\ref{for_seam}). In Figure \ref{fig_octagon}, one of the two octagons forming the four-holed sphere, we have
				\begin{equation}
						\sinh |AB|\sinh |BC| = \cos \angle O \text{ namely } \sinh \frac{c}{4}\sinh\frac{d}{2} = \cos \frac{\pi}{4}. 
						\label{for_seam}
				\end{equation}
			According to \cite[Lemma 2.5]{fortier2020local}, the cuff length of the four-holed spheres is approximately $6.980$. Then by (\ref{for_seam}), this proposition holds. 
		\end{proof}
			\begin{figure}[htbp]
					\begin{minipage}[t]{0.49\textwidth}
							\centering
							\includegraphics{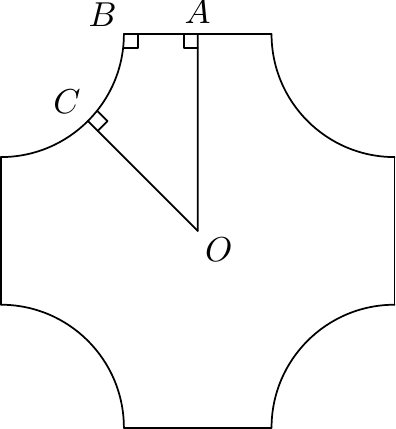}
							\caption{The right-angled octagon}
							\label{fig_octagon}
					\end{minipage}
					\begin{minipage}[t]{0.49\textwidth}
						\centering
						\includegraphics{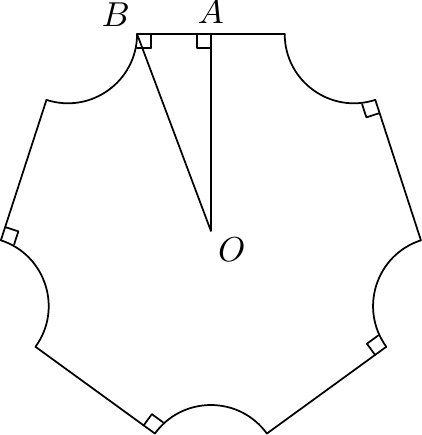}
						\caption{The polygon $Q$}
						\label{fig_octagon_2}
					\end{minipage}
				\end{figure}

		For the surface $S_g^1$ we have 
		\begin{proposition}
				The diameter of the surface $S_g^1$ satisfies 
				\[
						\diam(S_g^1)< 4\log\left( \frac{4g+4}{\pi} \right).
				\]
				\label{prop_diam_small}
		\end{proposition}
		\begin{proof}
				Recall that the surface $S_g^1$  consists of two $g+1$-holed spheres and each of the $g+1$-holed sphere consists of two right-angled regular ($2g+2$)-gons. For any $x,y\in S_g^1$, 
				for the two (possibly coicide) regular ($2g+2$)-polygons containing $x$ and $y$, there is a curve connecting $x$ and $y$, contained in the union of these two polygons (see Figure \ref{fig_diam}). Therefore, if we denote one of the four regular $2g+2$ polygons by $Q$, then 
				\begin{figure}[htbp]
						\centering
						\includegraphics{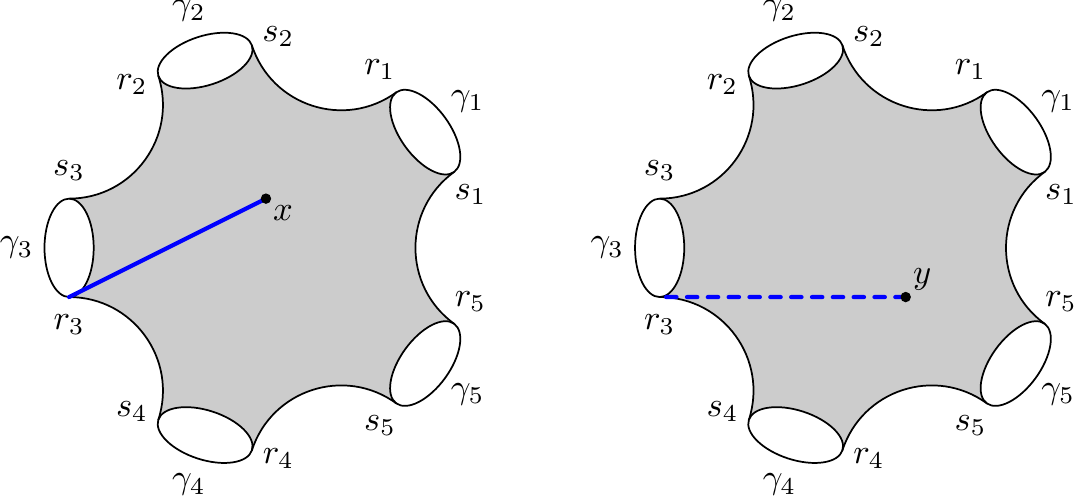}
						\caption{The path between $x$ and $y$}
						\label{fig_diam}
				\end{figure}
				\[
						\diam (S_g^1)\le 2\diam(Q), 
				\]

				The diameter of $Q$ is realized by $2|OB|$ in Figure \ref{fig_octagon_2}. 
				In the triangle $\triangle OAB$, by (\ref{for_tri_2}), 
				\[
						\cosh |OB| = \cot \angle O \cot \angle B
				\]
				namely
				\[
						\cosh |OB| = \cot \frac{\pi}{4} \cot \frac{\pi}{2g+2} = \cot \frac{\pi}{2g+2} < \frac{2g+2}{\pi}. 
				\]
				Therefore
\[
						\diam(S_g^1) \le 2\diam(Q) 
						\le 4 |OB|\\
						< 4\arccosh \left( \frac{2g+2}{\pi} \right)
						< 4\log\left( \frac{4g+4}{\pi} \right). 
\]
		\end{proof}

		Now we are ready to prove 
		\begin{theorem}
				When $g\ge 13$, 
				\[
						d_{\mathcal{T}}(S_g^1,S_g^3) > \frac{1}{2} \log(g-6) -\frac{1}{2}\log \left( \frac{40}{3}\log\left( \frac{4g+4}{\pi} \right) \right).
				\]
				\label{thm_dist_large}
		\end{theorem}
		\begin{proof}
				The proof here is similar to the proof of \cite[Lemma 5.1]{rafi2013diameter}. 

				We let $f:S_g^1\to S_g^3$ be a Lipschitz homeomorphism with $L(f)= d_L(S_g^1, S_g^3)$ (The existence of this homeomorphism is verified in \cite{thurston1998minimal}. ). By Proposition \ref{prop_diam_large}, we pick $x,y\in S_g^3$ with $d(x,y)\ge 0.6\left\lfloor \frac{g-5}{2}\right\rfloor$. By Proposition \ref{prop_diam_small}, $d(f^{-1}(x),f^{-1}(y))< 4\log\left( \frac{4g+4}{\pi} \right)$. Then
				\[
						L(f) \ge \frac{d(x,y)}{d(f^{-1}(x),f^{-1}(y))} > \frac{0.6\left\lfloor \frac{g-5}{2}\right\rfloor}{4\log\left( \frac{4g+4}{\pi} \right)}> \frac{3 (g-6)}{40\log\left( \frac{4g+4}{\pi} \right)}. 
				\]
				Hence 
				\[
						d_L(S_g^1,S_g^3) = \log L(f) > \log(g-6) -\log \left( \frac{40}{3}\log\left( \frac{4g+4}{\pi} \right) \right). 
				\]
				By (\ref{for_thurston_metric}), 
				\[
						d_{\mathcal{T}}(S_g^1,S_g^3)\ge \frac{1}{2}d_L(\Sigma_g^1,\Sigma_g^3) > \frac{1}{2} \log(g-6) -\frac{1}{2}\log \left( \frac{40}{3}\log\left( \frac{4g+4}{\pi} \right) \right).
				\]
		\end{proof}

 \bibliographystyle{alpha}
 \bibliography{systole, systole_euc, systole_dist}   

\end{document}